\documentclass[reqno,10pt]{amsart}

\usepackage{amsmath,amsfonts,amssymb,amsthm,epsfig,amscd}
\usepackage[]{fontenc}
\usepackage[latin1]{inputenc}
\usepackage{enumerate}
\usepackage[cmtip,all]{xy}
\usepackage{tabularx,multirow}
\usepackage[dvipsnames]{xcolor}
\usepackage[colorlinks=true]{hyperref}
\usepackage{url}
\usepackage{mathtools}
\usepackage{tikz}
\usepackage{tikz-cd}
\usepackage{enumitem}
\usepackage{float}

\usepackage{dynkin-diagrams}
\def\row#1/#2!{#1_{\IfStrEq{#2}{}{n}{#2}} & \dynkin{#1}{#2}\\}

\newlist{thmlist}{enumerate}{1}
\setlist[thmlist]{label=(\roman{thmlisti}), ref=\thethm.(\roman{thmlisti}),noitemsep}

 \allowdisplaybreaks \numberwithin{equation}{section}

\newtheorem{thm}{Theorem}[section]
\newtheorem*{thm*}{Theorem}

\newtheorem{prop}[thm]{Proposition}

\newtheorem{lem}[thm]{Lemma}

\theoremstyle{definition}
\newtheorem{defn}[thm]{Definition}

\theoremstyle{definition}
\newtheorem{rmk}[thm]{Remark}
\newtheorem{exmp}[thm]{Example}

\newcommand{\q}{\mathfrak{q}}
\newcommand{\m}{\mathfrak{m}}
\newcommand{\C}{\mathbb{C}}
\newcommand{\R}{\mathbb{R}}
\newcommand{\Z}{\mathbb{Z}}
\newcommand{\Ol}{\mathcal{O}}

\newcommand{\Q}{\mathbb{Q}}

\newcommand{\Ext}{\mathcal{E}\textit{xt}}

\newcommand{\Pic}{\textnormal{Pic}}

\newcommand{\ch}{\textnormal{ch}}
\newcommand{\td}{\textnormal{td}}

\newcommand{\NS}{\textnormal{NS}}

\newcommand{\rk}{\textnormal{rk}}

\newcommand{\Deg}{\textnormal{deg}}

\newcommand{\Dim}{\textnormal{dim}}
\newcommand{\Sym}{\textnormal{Sym}}
\newcommand{\id}{\textnormal{id}}

\begin{document}
	
	\title[Hodge classes of type $(2, 2)$ on Hilbert squares of projective K3 surfaces]{Hodge classes of type $(2, 2)$ on Hilbert squares of projective K3 surfaces}
	
	\author{Simone Novario}
	\address{Universit\`a degli Studi di Milano, Dipartimento di Matematica ``F. Enriques'', Via Cesare Saldini 50, 20133 Milano, Italy. \newline \indent Universit\'e de Poitiers, Laboratoire de Math\'ematiques et Applications, T\'el\'eport 2, Boulevard Marie et Pierre Curie, 86962 Futuroscope Chasseneuil, France.}
	
	\email{simone.novario@unimi.it, simone.novario@math.univ-poitiers.fr}
	
	\begin{abstract}
		We give a basis for the vector space generated by rational Hodge classes of type $(2, 2)$ on the Hilbert square of a projective K3 surface general in its rank, which is a subspace of the singular cohomology ring with rational coefficients: we use Nakajima operators and an algebraic model developed by Lehn and Sorger as main tools. We then obtain a basis of the lattice generated by integral Hodge classes of type $(2, 2)$ on the Hilbert square of a projective K3 surface general in its rank: we exploit lattice theory, a theorem by Qin and Wang and a result by Ellingsrud, Göttsche and Lehn.
	\end{abstract}
	
	\maketitle
	
	% only for working version
	%\tableofcontents
	
	%%%%%%%%%%%%%%%%%%%%%%%%%%%%%%%%%%%%%%%%%%
	%%%%%%%%%%%%%%%%%%%%%%%%%%%%%%%%%%%%%%%%%%
	%%%%%%%%%%%%%%%%%%%%%%%%%%%%%%%%%%%%%%%%%%
	
	\section{Introduction}\label{intro} 
	Irreducible holomorphic symplectic (IHS) manifolds are a generalization to higher dimensions of K3 surfaces. The interest in these varieties has been increasing thanks to the Beauville--Bogomolov decomposition theorem: up to a finite étale cover, any compact Kähler manifold with trivial first Chern class is the product of a complex torus, Calabi--Yau manifolds and IHS manifolds, see \cite[Théorème 2]{beauville1983varietes}. In this paper we deal with Hilbert schemes of $2$ points on a K3 surface, also known as \emph{Hilbert squares of K3 surfaces}, the first example of IHS manifold other than K3 surfaces to be found, see \cite{fujiki1987rham}: if $S$ is a K3 surface, we denote by $S^{[2]}$ its Hilbert square. There exists a quadratic form on the second cohomology group $H^2(S^{[2]}, \Z)$, called \emph{Beauville--Bogomolov--Fujiki} (BBF) form and denoted by\,\,$q_{S^{[2]}}$, so that $(H^2(S^{[2]}, \Z), q_{S^{[2]}})$ is a lattice.
	
	If $S$ is a K3 surface, we denote by $T(S)=\NS(S)^{\perp}$ the \emph{transcendental lattice} of $S$, where $\NS(S)$ is the Néron--Severi group of $S$ and the orthogonal is taken with respect to the intersection product. Moreover, $T(S)_{\Q}=T(S)\otimes \Q$ is a rational Hodge structure of weight $2$, and we denote by
	\begin{equation*}
	E_S:=\text{Hom}_0(T(S)_{\Q}, T(S)_{\Q})
	\end{equation*}
	the algebra of endomorphisms on $T(S)_{\Q}$ of weight $0$. See \cite[$\S 3$]{huybrechts2016lectures} for details on Hodge structures. Zahrin showed in \cite{Zarhin1983} that $E_S$ is either $\Q$, or a totally real field or a CM field. We say that $S$ is \emph{general in its rank}, or simply \emph{general}, if $E_S \cong \Q$. As remarked in \cite{van2008real}, see also \cite[$\S 1.1$]{elsenhans2016point}, K3 surfaces which are not general gives a locus of positive codimension in the analytic moduli space of K3 surfaces of fixed Picard rank, unless the Picard rank is $20$, in which case $E_S$ is always a CM field.
	
	In this paper we study the $\Q$-vector space of \emph{rational Hodge classes} of type $(2, 2)$ and the lattice of \emph{integral Hodge classes} of type $(2, 2)$ on the Hilbert square of a general projective K3 surface $S$, defined respectively as
	\begin{equation*}
		H^{2,2}(S^{[2]}, \Q):=H^4(S^{[2]}, \Q) \cap H^{2,2}(S^{[2]}), \qquad H^{2,2}(S^{[2]}, \Z):=H^4(S^{[2]}, \Z) \cap H^{2,2}(S^{[2]}).
	\end{equation*}
	The integral bilinear form considered on $H^{2,2}(S^{[2]}, \Z)$ is the cup product. Hodge classes are usually studied in the context of the so-called \emph{Hodge conjecture}: this states that given a smooth complex projective variety $Y$, the subspace of $H^{2k}(Y, \Q)$ generated by algebraic cycles, i.e., classes which are obtained as fundamental cohomological classes $[Z]$ of subvarieties $Z \subset Y$, coincides with $H^{k, k}(Y, \Q)$, which is by definition the set $H^{2k}(Y, \Q) \cap H^{k, k}(Y)$ of \emph{rational Hodge classes} of type $(k, k)$. Similarly one defines \emph{integral Hodge classes} of type $(k, k)$ as the elements which belong to the set $H^{k,k}(Y, \Z):=H^{2k}(Y, \Z) \cap H^{k,k}(Y)$. We now present the main result of this paper, which gives a basis of the lattice $H^{2,2}(S^{[2]}, \Z)$, where $S$ is a general projective K3 surface whose Picard group is known, cf.\,Theorem \ref{thm H^2,2(X, Z) generico}.
	\begin{thm*}
		Let $S$ be a general projective K3 surface and let $\{b_1, \dots , b_r\}$ be a basis of $\Pic(S)$. Then the lattice $H^{2,2}(S^{[2]}, \Z)$ is odd. Moreover, $\textnormal{rk}(H^{2,2}(S^{[2]}, \Z))=\frac{(r+1)r}{2}+r+2$, and a basis of $H^{2,2}(S^{[2]}, \Z)$ is the following:
		\begin{equation} \label{eq basis introduction}
			\left\{b_ib_j, \frac{b_i^2-b_i\delta}{2}, \frac{1}{8}\left(\delta^2+\frac{1}{3}c_2(S^{[2]})\right), \delta^2\right\}_{1 \le i \le j \le r,}
		\end{equation}
	where $c_2(S^{[2]}) \in H^{2,2}(S^{[2]}, \Z)$ is the second Chern class of $S^{[2]}$.
	\end{thm*}
	Clearly the elements in (\ref{eq basis introduction}) give a basis also for the $\Q$-vector space $H^{2,2}(S^{[2]}, \Q)$: a basis in terms of Nakajima operators is described in Theorem \ref{thm dim H2,2(X, Q) per K3 qualsiasi}. Topics discussed in this paper are related to \cite{shen2016hyperkahler}, where Shen studies Hyperkähler manifolds of Jacobian type.
	
	The paper is organised as follows. In Section \ref{section lattices} we recall some notions of lattice theory. In Section \ref{section IHS gen} we give the definition of IHS manifold, together with general properties. In Section \ref{section IHS K32 type} we present some useful results which hold for an IHS manifold $X$ of $K3^{[2]}$-type, in particular the relation between the BBF bilinear form and the cup product between elements in $H^4(X, \Q)$, and the definition of the dual $q_X^{\vee} \in H^{2,2}(X, \Q)$ of the BBF quadratic form. In Section\,\,\ref{section Nakajima} we introduce Nakajima operators, following \cite{nakajima1997heisenberg} and \cite{lehn1999chern}, and the algebraic model developed by Lehn and Sorger in \cite{lehn2003cup}. In Section\,\,\ref{section rational Hodge}, using Nakajima operators and the model of Lehn--Sorger, we prove Theorem\,\,\ref{thm dim H2,2(X, Q) per K3 qualsiasi}, which gives a basis of the $\Q$-vector space of rational Hodge classes of type $(2, 2)$ on the Hilbert square of a general projective K3 surface whose Picard group is known: the basis is presented in terms of Nakajima operators. In Section\,\,\ref{section second Chern}, using a result by Ellingsrud, Göttsche and Lehn in \cite{ellingsrud2001cobordism}, we obtain an explicit description of the second Chern class of the Hilbert square of a projective K3 surface depending on Nakajima operators, cf.\,Proposition\,\,\ref{prop c2(X) con operatori Nakajima}. We use this description in Section \ref{section integral Hodge}, where we pass to study integral Hodge classes of type $(2, 2)$ on the Hilbert square of a general projective K3 surface. The main result of this section is Theorem \ref{thm H^2,2(X, Z) generico}, which describes a basis for $H^{2,2}(S^{[2]}, \Z)$ for a general projective K3 surface\,\,$S$: the strategy of the proof is to combine Proposition \ref{prop c2(X) con operatori Nakajima} with a result by Qin and Wang in \cite{qin2005integral}, together with the property that the lattice $H^2(S, \Z)$ is unimodular.
	
	A \emph{generic K3 surface of degree} $2t$ is a general projective K3 surface $S_{2t}$ whose Picard group is generated by the class of an ample divisor $H$ with $H^2=2t$ with respect to the intersection form. This paper is based on Chapter 3 of the author's PhD thesis, where Theorem \ref{thm H^2,2(X, Z) generico} is used to show that Hilbert squares of general K3 surfaces of degree\,\,$2t$, with $t \neq 2$, admitting an ample divisor $D$ with $q_X(D)=2$ are double EPW sextics, see \cite{o2006irreducible} for the definition, where\,\,$q_X$ represents the BBF quadratic form on $X$. Proofs and details that will be omitted in this paper can be found in the author's PhD thesis, see \cite{novario2021ths}.
	%"But that's another story for another day..."
	
	\textbf{Acknowledgements:} I thank Samuel Boissière for his constant support and for useful clarifications on the theory of Nakajima operators. I thank Bert van Geemen for many useful discussions. I thank Simon Brandhorst for precious clarifications about lattice theory and the software Sage. I deeply thank Olivier Benoist and Daniel Huybrechts for useful comments, summarised in Remark \ref{rmk general K3}, on a first draft of this paper, which improved the statement of the main Theorem. I also thank Pietro Beri and Ángel David Ríos Ortiz for many useful discussions.
	\section{Lattices} \label{section lattices}
	In this section we recall the most important definitions and results of lattice theory that we need. General references are \cite{nikulin1980integral} and \cite{conway2013sphere}.
	\begin{defn} \label{def lattice}
		A \emph{lattice} $L$ is a free $\Z$-module of finite rank together with a symmetric bilinear form	$b: L \times L \rightarrow \Z$. We denote by $q: L \rightarrow \Z$ the quadratic form defined by $q(x):=b(x, x)$ for every $x \in L$.
	\end{defn}
If $L$ is a lattice of rank $n$ and $\mathcal{B}:=\{e_1, \dots , e_n\}$ is a $\Z$-basis of $L$, we call \emph{Gram matrix} of $L$ associated to $\mathcal{B}$ the following symmetric matrix:
	\begin{equation*}
		\begin{pmatrix}
			b(e_1, e_1) & \cdots & b(e_1, e_n) \\
			\vdots & \ddots & \vdots \\
			b(e_n, e_1) & \cdots & b(e_n, e_n)
		\end{pmatrix}.
	\end{equation*}
	A \emph{non-degenerate} lattice is a lattice $L$ of rank $n$ such that for any non-zero $l \in L$ there exists $l^{\prime} \in L$ such that $b(l, l^{\prime}) \neq 0$, equivalently, $\text{det}(G) \neq 0$ if $G$ is a Gram matrix of $L$. A lattice $L$ is \emph{even} if $b(l, l) \in 2\Z$ for every $l \in L$, and \emph{odd} if it is not even. A \emph{sublattice} of a lattice $L$ is a free submodule $L^{\prime} \subseteq L$ with symmetric bilinear form $b^{\prime}:=b|_{L^{\prime} \times L^{\prime}}$. A sublattice $L^{\prime} \subseteq L$ is \emph{primitive} if $L/L^{\prime}$ is a free module. The \emph{direct sum} of two lattices $L_1$ and $L_2$ is by definition the lattice $L_1 \oplus L_2$ whose bilinear form is $b(v_1+v_2, w_1+w_2):=b_1(v_1, w_1)+b_2(v_2, w_2)$ for every $v_1, w_1 \in L_1$ and $v_2, w_2 \in L_2$, where $b_1$ and $b_2$ are the bilinear forms of $L_1$ and $L_2$ respectively.
	
	For a lattice $L$ of rank $n$ we write $L_{\R}:=L \otimes_{\Z} \R$ and we extend $\R$-bilinearly the bilinear form $b$ to $L_{\R}$, similarly we extend $q$ to $L_{\R}$. If the lattice is non-degenerate, the \emph{signature} of $L$ is the signature $(l_{(+)}, l_{(-)})$ of the quadratic form on $L_{\R}$. A non-degenerate lattice is \emph{positive definite} if $l_{(-)}=0$, similarly it is \emph{negative definite} if $l_{(+)}=0$, while it is \emph{indefinite} if $l_{(+)}, l_{(-)} \neq 0$. 
	
	 The \emph{determinant} of a lattice $L$ is the determinant of a Gram matrix $G$ of the lattice, and the \emph{discriminant} of $L$ is $\text{disc}(L):=|\text{det}(G)|$. A \emph{unimodular} lattice is a lattice $L$ such that $\text{disc}(L)=1$: if $L$ is a unimodular lattice, for every $x \in L$ there exists $y \in L$ such that $b(x, y)=1$. Let $k$ be an integer: the easiest example of lattice is $\langle k \rangle$, which is the rank one lattice $L=\Z e$ with bilinear form $b(e, e)=k$. Other two important examples of lattices are the \emph{hyperbolic lattice}, which is denoted by $U$, and $E_8(-1)$: the former is the unique unimodular lattice of rank $2$ and signature $(1, 1)$, the latter is an even unimodular lattice of signature $(0, 8)$. The following two matrices are Gram matrices respectively for $U$ and $E_8(-1)$:
	\begin{equation} \label{eq Gram matrices}
		\begin{pmatrix}
			0 & 1 \\
			1 & 0
		\end{pmatrix}, \qquad
		\begin{pmatrix}
		-2 & 1 &  &  &  &  &  &  \\
		1 & -2 & 1 &  &  &  &  &  \\
		& 1 & -2 & 1 &  &  &  & 1 \\
		&  & 1 & -2 & 1 &  &  &  \\
		&  &  & 1 & -2 & 1 &  &  \\
		&  &  &  & 1 & -2 & 1 &  \\
		&  &  &  &  & 1 & -2 &  \\
		&  & 1 &  &  &  &  & -2 \\
		\end{pmatrix}.
	\end{equation}
	A \emph{morphism of lattices} $\varphi: L \rightarrow L^{\prime}$ between two lattices $L$ and $L^{\prime}$ whose bilinear forms are respectively $b$ and $b^{\prime}$ is a morphism of $\Z$-modules such that for every $l_1, l_2 \in L$ we have $b(l_1, l_2)=b^{\prime}(\varphi(l_1), \varphi(l_2))$. Morphisms between two non-degenerate lattices are necessarily injective. We call \emph{isometry} a bijective morphism of lattices. A lattice $L$ \emph{embeds primitively} in a lattice $L^{\prime}$ if there is a morphism $\varphi: L \rightarrow L^{\prime}$ such that $\varphi(L)$ is a primitive sublattice of $L^{\prime}$. 
	\section{Generalities on IHS manifolds} \label{section IHS gen}
	In this section we recall basics on irreducible holomorphic symplectic manifolds.
	\begin{defn}
		An \emph{irreducible holomorphic symplectic} (IHS) \emph{manifold} is a simply connected compact complex Kähler manifold $X$ such that $H^0(X, \Omega^2_X)$ is generated by a non-degenerate holomorphic $2$-form, called \emph{symplectic form}.
	\end{defn}
	The dimension of an IHS manifold is necessarily even as a consequence of the existence of a symplectic form. The Enriques--Kodaira classification of compact complex surfaces shows that the only IHS manifolds of dimension $2$ are K3 surfaces. If\,\,$X$ is an IHS manifold, then the $\C$-vector space $H^0(X, \Omega^p_X)$ is zero if $p$ is odd, and it is generated by\,\,$\sigma^{\frac{p}{2}}$ if $0 \le p \le \Dim(X)$ is even, where\,\,$\sigma$ is a symplectic form, see \cite[Proposition 3]{beauville1983varietes}. The Picard group $\Pic(X)$ is isomorphic to the Néron--Severi group $\NS(X)=H^{1,1}(X)_{\R} \cap H^2(X, \Z)$, and this embeds in the second cohomology group $H^2(X, \Z)$. By the universal coefficient theorem, the second singular cohomology group $H^2(X, \Z)$ is torsion free. Moreover, by a result due to Beauville, Bogomolov and Fujiki, $H^2(X, \Z)$ can be equipped with a non-degenerate integral quadratic form, denoted by $q_X$ and called \emph{Beauville--Bogomolov--Fujiki} (BBF) form, see \cite{beauville1983varietes} and \cite{fujiki1987rham} for details: in particular $(H^2(X, \Z), q_X)$ has a structure of even lattice of signature $(3, b_2(X)-3)$, where $b_2(X)$ is the second Betti number of\,\,$X$. For K3 surfaces the BBF form coincides with the intersection form. The \emph{transcendental lattice} of an IHS manifold $X$ is defined as $T(X):=(\NS(X))^{\perp}$: the orthogonal is taken with respect to the BBF form.
	
	Let $S$ be a K3 surface: we denote by $S^{[n]}$ the \emph{Hilbert scheme of} $n$ \emph{points on} $S$, which is the scheme which parametrises zero-dimensional closed subschemes of length $n$ on $S$. By \cite[Théorème 3]{beauville1983varietes} the variety $S^{[n]}$ is an IHS manifold of dimension $2n$. If $S^{(n)}$ is the quotient of $S^n=S \times \dots \times S$ by the symmetric group of $n$ elements, the morphism $\rho: S^{[n]} \rightarrow S^{(n)}$ defined by $\rho([\xi])=\sum_x l(\Ol_{\xi, x}) \cdot x$ is called \emph{Hilbert--Chow morphism}, see for instance \cite{iversen2006linear}. In particular $\rho$ is a desingularisation of $S^{(n)}$, whose singular locus is the so-called diagonal, i.e., the set of cycles $p_1+ \dots + p_n$ such that there exist $i$ and $j$ with $i \neq j$ and $p_i=p_j$. The pre-image of the diagonal in $S^{(n)}$ is an irreducible divisor $E$ on $S^{[n]}$, and there exists a primitive class $\delta \in \Pic(S^{[n]})$ such that $2\delta=[E]$. An important fact is the following: there exists a primitive embedding of lattices 
	\begin{equation*}
		i: H^2(S, \Z) \hookrightarrow H^2(S^{[n]}, \Z)
	\end{equation*}
	such that $H^2(S^{[n]}, \Z)=i(H^2(S, \Z)) \oplus \Z\delta$, and $q_{S^{[n]}}(\delta)=-2(n-1)$. For a K3 surface $S$ there exists an isometry of lattices $H^2(S, \Z) \cong U^{\oplus 3} \oplus E_8(-1)^{\oplus 2}$, see for instance \cite[$\S \text{VII.3}$]{barth2015compact}, in particular $H^2(S, \Z)$ is an even unimodular lattice of signature $(3, 19)$. Hence there is an isometry of lattices
	\begin{equation*}
		H^2(S^{[n]}, \Z) \cong U^{\oplus 3} \oplus E_8(-1)^{\oplus 2} \oplus \langle -2(n-1) \rangle,
	\end{equation*}
	similarly $\Pic(S^{[n]}) =i(\Pic(S)) \oplus \Z \delta$: see \cite[$\S 6$]{beauville1983varietes} for details. Moreover, the singular cohomology ring $H^*(S^{[n]}, \Z)$ for a K3 surface $S$ and $n \ge 1$ is torsion free by \cite[Theorem 1]{markman2007integral}. When $n=2$, we call $S^{[2]}$ the \emph{Hilbert square of} $S$. An IHS manifold which is deformation equivalent to the Hilbert square of a K3 surface is said to be of $K3^{[2]}$-\emph{type}.
	
	The other known examples of IHS manifolds up to deformation equivalence are \emph{generalised Kummer varieties}, see \cite[$\S 7$]{beauville1983varietes}, an isolated example of dimension $10$ and second Betti number $b_2=24$, see \cite{o1999desingularized}, and an isolated example of dimension $6$ and second Betti number $b_2=8$, see \cite{o2003new}. We do not discuss details on these examples since in this paper we deal only with Hilbert squares of K3 surfaces.
	\section{Generalities on IHS manifolds of $K3^{[2]}$-type} \label{section IHS K32 type}
	Let $X$ be an IHS manifold of dimension $4$ of $K3^{[2]}$-type. In this section we recall some useful properties from \cite[$\S 2$]{o2008irreducible}, in particular the link between the intersection pairing on $H^4(X, \Q)$ and the BBF form on $H^2(X, \Q)$.
	
	First of all, we state the following corollary of Verbitsky's results in \cite{verbitsky1996cohomology}, obtained by Guan in \cite{guan2001betti}, see also \cite[Corollary 2.5]{o2010higher}. We denote by $b_i(X)$ the $i$-th Betti number of $X$.
	
	\begin{prop} \label{prop Guan}
		Let $X$ be an IHS manifold of dimension $4$. Then $b_2(X) \le 23$. If equality holds then $b_3(X)=0$ and the map
		\begin{equation*}
			\textnormal{Sym}^2H^2(X, \Q) \rightarrow H^4(X, \Q)
		\end{equation*}
	induced by the cup product is an isomorphism. In particular this happens when $X$ is an IHS fourfold of $K3^{[2]}$-type.
	\end{prop}
We denote by $\langle \, \cdot \, , \, \cdot \, \rangle$ the intersection pairing on the singular cohomology group $H^4(X, \Z)$ induced by the cup product and defined by $\langle \alpha, \beta \rangle :=\int_X \alpha\beta$. We use the same notation $\langle \, \cdot \, , \, \cdot \, \rangle$ for the $\Q$-bilinear extension of the intersection pairing above to $H^4(X, \Q)$, obtaining a $\Q$-valued intersection pairing on $\Sym^2H^2(X, \Q)$. Let $X$ be an IHS manifold of $K3^{[2]}$-type: then the intersection pairing $\langle \, \cdot \, , \, \cdot \, \rangle$ can be described in terms of the $\Q$-extension of the BBF bilinear form on $H^2(X, \Q)$ thanks to the following Proposition, see \cite[Remark 2.1]{o2008irreducible}.
\begin{prop}[O'Grady] \label{prop rmk 2.1 OG}
Let $X$ be an IHS fourfold of $K3^{[2]}$-type. The intersection pairing $\langle \, \cdot \, , \, \cdot \, \rangle$ defined above coincides with the bilinear form on $\textnormal{Sym}^2H^2(X, \Q)$ given by
\begin{equation*}
	\langle \alpha_1\alpha_2, \alpha_3\alpha_4 \rangle =(\alpha_1, \alpha_2)(\alpha_3, \alpha_4)+(\alpha_1, \alpha_3)(\alpha_2, \alpha_4)+(\alpha_1, \alpha_4)(\alpha_2, \alpha_3)
\end{equation*}
for every $\alpha_1, \alpha_2, \alpha_3, \alpha_4 \in H^2(X, \Q)$, where $(\,\cdot \, , \, \cdot \, )$ denotes the BBF bilinear form on $H^2(X, \Q)$.
\end{prop}
Let $q_X$ be the BBF quadratic form on $X$ and $\{e_1, \dots , e_{23}\}$ be a basis of $H^2(X, \Q)$. Let $\{e_1^{\vee}, \dots , e_{23}^{\vee}\}$ be the dual basis in $H^2(X, \Q)^{\vee}$, i.e., $e_i^{\vee}(e_j)=\delta_{i,j}$. Then we have 
\begin{equation*}
	q_X=\displaystyle\sum_{i,j}g_{i,j}e_i^{\vee}\otimes e_j^{\vee}, \qquad q_X^{\vee}=\displaystyle\sum_{i,j}m_{i,j}e_ie_j,
\end{equation*}
where $g_{i,j}:=(e_i, e_j)$, the matrix $(g_{i,j})$ is symmetric and $(m_{i,j})=(g_{i,j})^{-1}$. As shown in \cite[Proposition 2.2]{o2008irreducible}, the products  $\langle q_X^{\vee}, \alpha \rangle$ for every $\alpha \in H^4(X, \Q)$ can be computed exploiting the BBF bilinear form on $H^2(X, \Q)$.
\begin{prop}[O'Grady] \label{prop O'Grady int q_Xvee}
Let $X$ be an IHS fourfold of $K3^{[2]}$-type. Let $\langle \, \cdot \, , \, \cdot \, \rangle$ be the bilinear form described in Proposition $\ref{prop rmk 2.1 OG}$. Then $\langle \, \cdot \, , \, \cdot \, \rangle$ is non-degenerate and 
\begin{equation*}
	\begin{array}{l}
		\langle q_X^{\vee}, \alpha \beta \rangle = 25(\alpha, \beta) \qquad \text{for all}\,\,\alpha, \beta \in H^2(X, \Q), \\[1ex]
		\langle q_X^{\vee}, q_X^{\vee} \rangle =23 \cdot 25.
	\end{array}
\end{equation*} 
\end{prop}
O'Grady has shown in \cite[$\S 3$]{o2008irreducible} that $q_X^{\vee}$ is a rational multiple of $c_2(X)$, the second Chern class of the tangent bundle of $X$, in particular it is an element of $H^{2,2}(X, \Q)$, i.e., it is a rational Hodge class of type $(2, 2)$. 
\begin{prop}[O'Grady] \label{prop O'Grady c_2(X)}
Let $X$ be an IHS fourfold of $K3^{[2]}$-type. Then $q_X^{\vee} \in H^{2,2}(X, \Q)$, i.e., $q_X^{\vee}$ is a rational Hodge class of $X$ of type $(2, 2)$, and
\begin{equation*}
	\frac{6}{5}q_X^{\vee}=c_2(X) \in H^{2,2}(X, \Z).
\end{equation*}
Moreover, $\frac{2}{5}q_X^{\vee} \in H^{2,2}(X, \Z)$ is an integral Hodge class of $X$ of type $(2, 2)$.
\end{prop}
	\section{Nakajima operators and the Lehn--Sorger model} \label{section Nakajima}
In this section we introduce Nakajima operators following \cite{nakajima1997heisenberg} and \cite{lehn1999chern}, and the algebraic model developed by Lehn and Sorger in \cite{lehn2003cup}. 

Let $S$ be a smooth complex projective surface and let $S^{[n]}$ be the Hilbert scheme of $n$ points on $S$ for any integer $n>0$. We define
\begin{equation*}
	\mathbb{H}_n^S:=\displaystyle\oplus_{i=0}^{4n}H^i(S^{[n]}, \Q), \qquad \mathbb{H}^S:=\displaystyle\oplus_{n \ge 0} \mathbb{H}^S_n.
\end{equation*}
The unit of the $\Q$-algebra, with cup product, $\mathbb{H}^S_0 \cong \Q$ is called \emph{vacuum vector} and it is denoted by $|0\rangle$. If $1_{S^{[n]}}$ denotes the unit for the cup product in $\mathbb{H}_n^S$ for every $n$, the unit in $\mathbb{H}^S$ for the cup product is given by $|1\rangle :=\sum_{n \ge 0} 1_{S^{[n]}}$. The space $\mathbb{H}^S$ is double graded by $(n, i)$: we say that $n$ is the \emph{conformal weight} and $i$ is the \emph{cohomological degree}, denoted by $|\,\cdot\,|$. Let $\mathfrak{f} \in \text{End}(\mathbb{H}^S)$ be a linear operator. We say that $\mathfrak{f}$ is \emph{homogeneous of bidegree} $(\nu, \iota)$ if for any $n$ and $i$ we have $\mathfrak{f}(H^i(S^{[n]}, \Q)) \subset H^{i+\iota}(S^{[n+\nu]}, \Q)$. The \emph{commutator} of two homogeneous operators $\mathfrak{f}, \mathfrak{g} \in \text{End}(\mathbb{H}^S)$ is defined by
\begin{equation*}
	[\mathfrak{f}, \mathfrak{g}]:=\mathfrak{f} \circ \mathfrak{g} - (-1)^{|\mathfrak{f}|\cdot |\mathfrak{g}|} \mathfrak{g} \circ \mathfrak{f}.
\end{equation*}
We now define an intersection pairing $\langle \, \cdot \, , \, \cdot \, \rangle$ on $\mathbb{H}^S$. First of all, fix an integer $n \ge 0$ and let $\alpha, \beta \in \mathbb{H}^S_n$: we set
\begin{equation} \label{eq int pairing}
	\langle \alpha, \beta \rangle:=\displaystyle\int_{S^{[n]}}\alpha \beta
\end{equation}
Note that $\langle \alpha, \beta \rangle =0$ if $|\alpha|+|\beta| \neq 4n$. Then $\langle \, \cdot \, , \, \cdot \, \rangle$ extends naturally to a non-degenerate graded symmetric bilinear form on $\mathbb{H}^S$, which we denote again by $\langle \, \cdot \, , \, \cdot \, \rangle$. If $\mathfrak{f} \in \text{End}(\mathbb{H}^S)$ is a homogeneous operator, we define the \emph{adjoint operator} $\mathfrak{f}^{\dagger}$ as the homogeneous operator characterised by the relation
\begin{equation*}
	\langle \mathfrak{f}(\alpha), \beta \rangle = (-1)^{|\mathfrak{f}|\cdot |\alpha|}\langle \alpha, \mathfrak{f}^{\dagger}(\beta)\rangle.
\end{equation*}
Following \cite[$\S 1.2$]{lehn1999chern}, we define an irreducible subvariety $S^{[n, n+k]}$ of $S^{[n]}\times S \times S^{[n+k]}$ for any integers $n \ge 0, k > 0$. Let $X^{[n, n+k]} \subset S^{[n]} \times S^{[n+k]}$ be the uniquely determined closed subscheme with the property that any morphism $f=(f_1, f_2): T \rightarrow S^{[n]} \times S^{[n+k]}$ from an arbitrary variety $T$ factors through $X^{[n, n+k]}$ if and only if the following holds:
\begin{equation*}
	(f_1 \times \text{id}_S)^{-1}(\Xi_n^S) \subset (f_2 \times \text{id}_S)^{-1}(\Xi^S_{n+k}),
\end{equation*}
where $\Xi_n^S \subset S^{[n]} \times S$ is the universal family of subschemes parametrised by $S^{[n]}$. Closed points in $X^{[n, n+k]}$ corresponds to pairs $(\xi, \xi^{\prime})$ of subschemes with $\xi \subseteq \xi^{\prime}$. Then one obtains a morphism similar to the Hilbert--Chow morphism:
\begin{equation*}
	\rho: X^{[n, n+k]} \rightarrow \text{Sym}^kS.
\end{equation*}
We set $X_0^{[n, n+k]}:=\rho^{-1}(\Delta)$, where $\Delta \subset \Sym^kS$ is the small diagonal, and we consider on $X_0^{[n, n+k]}$ the reduced induced subscheme structure. Set-theoretically we can identify $X_0^{[n, n+k]}$ with the following subset of $S^{[n]}\times S \times S^{[n+k]}$:
\begin{equation*}
	X_0^{[n, n+k]}:=\{(\xi, x, \xi^{\prime})\,|\, \xi \subseteq \xi^{\prime}\,\,\text{and}\,\,\text{Supp}(\mathcal{I}_{\xi}/\mathcal{I}_{\xi^{\prime}})=x\},
\end{equation*}
where $\mathcal{I}_{\xi}$ is the ideal sheaf of $\xi$. We define $S^{[n, n+k]}$ as
\begin{equation*}
	S^{[n, n+k]}:=\overline{\{(\xi, x, \xi^{\prime}) \in X_0^{[n, n+k]}\, | \, l(\xi_x)=0\}},
\end{equation*}
where the closure is the Zariski closure. Then $S^{[n, n+k]}$ is an irreducible subvariety of $S^{[n]} \times S \times S^{[n+k]}$ of complex dimension $2n+k+1$ by \cite[Lemma 1.1]{lehn1999chern}. Consider the following diagram
\begin{equation} \label{diagram Nakajima projections}
	\begin{tikzcd}
		& S^{[n]} \times S \times S^{[n+k]} \arrow{dl}[swap]{\varphi}  \arrow[d, "\rho"] \arrow[dr, "\psi"] & \\ 
		S^{[n]} & S & S^{[n+k]},
	\end{tikzcd}
\end{equation}
where $\varphi$, $\rho$ and $\psi$ are the projections respectively on $S^{[n]}$, $S$ and $S^{[n+k]}$.
\begin{defn} \label{defn Nakajima operators}
	Let $S$ be a smooth complex projective surface. The \emph{Nakajima creation operators}, known also as \emph{Heisenberg operators}, are defined as follows: for $\alpha \in H^*(S, \Q)$ and $x \in H^*(S^{[n]}, \Q)$ we set for $k > 0$
	\begin{equation} \label{eq defn Nakajima op}
		\q_k: H^*(S, \Q) \rightarrow \text{End}(\mathbb{H}^S), \qquad \q_k(\alpha)(x):=\psi_*\left( PD^{-1} \left[ S^{[n, n+k]} \right] \cdot \varphi^*(x) \cdot \rho^*(\alpha) \right),
	\end{equation}
where $[S^{[n, n+k]}] \in H_{4n+2k+2}(S^{[n]} \times S \times S^{[n+k]}, \Q)$ is the fundamental class of $S^{[n, n+k]}$, we denote by $PD$ the Poincaré duality, the dot is the cup product, and $\psi_*$ is the Gysin homomorphism, i.e., $\psi_*=PD^{-1}\psi_*PD$, where $\psi_*$ in the right-hand side is the pushforward in homology. By convention $\q_0=0$, and the \emph{Nakajima annihilation operators} are defined as
\begin{equation*}
	\q_k(\alpha):=(-1)^{-k}\q_{-k}(\alpha)^{\dagger} \qquad \text{for all}\,\,k<0.
\end{equation*}
\end{defn}
Note that the unit $|1\rangle \in \mathbb{H}^S$ can be represented as $|1\rangle =\sum_{n \ge 0} 1_{S^{[n]}}=\text{exp}(\q_1(1_S))|0\rangle$, which gives
\begin{equation} \label{formula 1sn Nakajima}
	1_{S^{[n]}}=\frac{1}{n!}\q_1(1)^n|0\rangle.
\end{equation}
The following commutation formula was obtained by Nakajima in \cite{nakajima1997heisenberg}.
\begin{thm}[Nakajima] \label{thm Nakajima}
	The operators $\q_i$ satisfy the following commutation formula:
	\begin{equation*}
		[\q_i(\alpha), \q_j(\beta)] =i \cdot \delta_{i+j, 0} \cdot \displaystyle\int_S \alpha\beta \cdot \textnormal{id}_{\mathbb{H}^S},
	\end{equation*}
where $\delta$ is the Kronecker delta.
\end{thm}
We now introduce the \emph{boundary operator}. Fix a positive integer $n > 0$ and let $\lambda =\{\lambda_1 \ge \lambda_2 \ge \dots \ge \lambda_s > 0\}$ be a partition of $n$, i.e., an $s$-uple of ordered positive integers such that $\lambda_1+\lambda_2+ \dots + \lambda_s=n$. Consider the Hilbert--Chow morphism $\rho: S^{[n]} \rightarrow \text{Sym}^nS$. We set
\begin{equation*}
	\text{Sym}^n_{\lambda}S:=\{\alpha \in \text{Sym}^nS\,|\, \alpha = \displaystyle\sum_{1 \le i \le s} \lambda_ix_i, \, \, x_i \in S \, \, \text{pairwise distinct}\}
\end{equation*}
for a fixed partition $\lambda$. By a theorem of Briançon $S^{[n]}_{\lambda}:=\rho^{-1}(\text{Sym}^n_{\lambda}S)$ is irreducible of dimension $n+s$, see \cite{brianccon1977description}. For example, $S^{[n]}_{(1, 1, \dots 1)}$ is the open subset of $S^{[n]}$ which corresponds to the configuration space of unordered $n$-tuples of pairwise distinct points: it is the only stratum which is open. The only stratum of codimension $1$ is $S^{[n]}_{(2, 1, \dots , 1)}$.
\begin{defn}
	The \emph{boundary} of $S^{[n]}$ for $n \ge 2$ is the irreducible divisor
	\begin{equation*}
		\partial S^{[n]}:=\displaystyle \bigcup_{\lambda \neq (1, 1, \dots , 1)} S^{[n]}_{\lambda}= \overline{S^{[n]}_{(2, 1, \dots , 1)}}.
	\end{equation*}
\end{defn}
\begin{rmk} \label{rmk boundary empty}
	For $n=1$ the boundary $\partial S^{[n]}$ is the empty set.
\end{rmk}
There is a description of the divisor $\partial S^{[n]}$ in sheaf theoretic terms. Consider the universal family of subschemes $\Xi_n \subset S^{[n]} \times S$. Let $p: \Xi_n \rightarrow S^{[n]}$ be the projection on the first factor. Since $p$ is a flat morphism of finite degree $n$, we have that $\Ol_S^{[n]}:=p_*(\Ol_{\Xi_n})$ is a locally free sheaf of rank $n$. Then $[\partial S^{[n]}]=-2c_1(\Ol_S^{[n]})$ by \cite[Lemma 3.7]{lehn1999chern}.

We can now define the boundary operator and the \emph{derivative} of a linear operator $\mathfrak{f} \in \text{End}(\mathbb{H}^S)$.
\begin{defn} \label{Defn boundary op}
	The \emph{boundary operator} $\mathfrak{d}: \mathbb{H^S} \rightarrow \mathbb{H^S}$ is the homogeneous linear map of bidegree $(0, 2)$ given by
	\begin{equation*}
		\mathfrak{d}(x):=c_1(\Ol_S^{[n]}) \cdot x=-\frac{1}{2}[\partial S^{[n]}]\cdot x \qquad \text{for all}\,\, x \in H^*(S^{[n]}, \Q).
	\end{equation*}
For any endomorphism $\mathfrak{f} \in \text{End}(\mathbb{H}^S)$ we define the \emph{derivative} of $\mathfrak{f}$ as
\begin{equation*}
	\mathfrak{f}^{\prime}:=[\mathfrak{d}, \mathfrak{f}]=\mathfrak{d} \circ \mathfrak{f}-\mathfrak{f}\circ \mathfrak{d}.
\end{equation*}
We denote by $\mathfrak{f}^{(n)}$ the higher derivatives.
\end{defn}
The following relation was obtained by Lehn, see \cite[Theorem 3.10, Proposition 3.12]{lehn1999chern}.
\begin{thm}[Lehn] \label{thm main Lehn}
	Let $S$ be a smooth complex projective surface. Then for any integers $n, m$ such that $n+m \neq 0$ and cohomology classes $\alpha, \beta \in H^*(S, \Q)$ we have
	\begin{equation*}
		[\q_n^{\prime}(\alpha), \q_m(\beta)]=-nm\cdot \q_{n+m}(\alpha\beta).
	\end{equation*}
\end{thm}
We now state the \emph{Ellingsrud--Göttsche--Lehn} formula. Let $\Xi_n \subset S^{[n]} \times S$ be the universal family of subschemes parametrised by $S^{[n]}$. By \cite[Theorem 1.9]{lehn1999chern} we have that $S^{[n, n+1]}$ is a smooth irreducible variety isomorphic to the blow-up of $S^{[n]} \times S$ along the universal family, i.e., $S^{[n, n+1]} \cong \text{Bl}_{\Xi_n^S}(S^{[n]} \times S)$. We get the following diagram
\begin{equation} \label{diagram EGL}
	\begin{tikzcd}
		& S^{[n, n+1]} \arrow[bend right=10,swap]{ddl}{\varphi}   \arrow[bend left=10]{ddr}{\rho} \arrow[d, "\sigma"]  \arrow[r, "\psi"] & S^{[n+1]} \\ 
		& S^{[n]} \times S \arrow[dl, "p"] \arrow{dr}[swap]{q} & \\
		S^{[n]} & \Xi_n \arrow[u, hook, "\iota"] & S,
	\end{tikzcd}	
\end{equation} 
where $\sigma$ is the blow-up of $S^{[n]} \times S$ in $\Xi_n^S$, the morphism $\iota$ is the inclusion and the other maps are the projections. If $N$ is the exceptional divisor of $\sigma$, let $\mathcal{L}:=\Ol_{S^{[n, n+1]}}(-N)$. Let $\mathcal{T}_n$ be the tangent bundle $\mathcal{T}S^{[n]}$ and $\omega_S$ be the canonical bundle of $S$. Given a smooth projective variety $X$ and $F \in \text{Coh}(X)$, following \cite[$\S 1.1$]{huybrechts2010geometry} we define the \emph{dual} $F^{\vee}$ as
\begin{equation} \label{eq defn dual}
	F^{\vee}:=\displaystyle\sum_i (-1)^i \Ext^{\,i}(F, \Ol_X).
\end{equation}
Let $K(X)$ be the Grothendieck group of vector bundles on $X$. We denote by $f^!$ the pullback between the Grothendieck groups of a morphism $f$ of smooth projective varieties, see \cite[$\S 15.1$]{fulton2013intersection} for more details. We can now state the \emph{Ellingsrud--Göttsche--Lehn formula}, see \cite[Proposition 2.3]{ellingsrud2001cobordism}.
\begin{prop}[EGL formula] \label{prop EGL}
	Keep notation as above. The following relation holds in $K(S^{[n, n+1]})$:
	\begin{equation*}
		\begin{array}{lll}
			\psi^{!} \mathcal{T}_{n+1} & = & \varphi^{!}\mathcal{T}_n + \mathcal{L} - \mathcal{L}\cdot \sigma^{!}(\Ol^{\vee}_{\Xi_n})+\mathcal{L}^{\vee}\cdot \rho^{!}\omega_S^{\vee} \\[1.2ex]
			& & -\mathcal{L}^{\vee}\cdot \sigma^{!}(\Ol_{\Xi_n})\cdot \rho^{!}\omega^{\vee}_S-\rho^{!}(\Ol_S-\mathcal{T}_S+\omega_S^{\vee}).
		\end{array}
	\end{equation*}
\end{prop}
Let now $S$ be a projective K3 surface. The following theorem by Qin and Wang gives integral basis of $H^2(S^{[2]}, \Z)$ and $H^4(S^{[2]}, \Z)$ in terms of Nakajima operators, see \cite[p.17]{boissiere2013smith}. See \cite[Theorem 5.4, Remark 5.6]{qin2005integral} for a more general statement.
\begin{thm}[Qin--Wang] \label{thm Qin--Wang}
	Let $S$ be a projective K3 surface and $X=S^{[2]}$ be its Hilbert square. Let $\{\alpha_i\}_{i=1, \dots , 22}$ be an integral basis of $H^2(S, \Z)$. Denote by $1 \in H^0(S, \Z)$ the unit and by $x \in H^4(S, \Z)$ the class of a point.
	\begin{enumerate}[label=(\roman*)]
		\item An integral basis of $H^2(X, \Z)$ in terms of Nakajima operators is given by
		\begin{equation*}
			\left\{\frac{1}{2}\q_2(1)|0 \rangle, \q_1(1)\q_1(\alpha_i)| 0 \rangle\right\}_{i=1, \dots , 22.}
		\end{equation*}
		\item An integral basis of $H^4(X, \Z)$ in terms of Nakajima operators is given by the following elements:
		\begin{equation*}
			\begin{array}{c}
				\q_1(1)\q_1(x)|0\rangle, \qquad \q_2(\alpha_i)| 0 \rangle, \qquad \q_1(\alpha_i)\q_1(\alpha_j)| 0 \rangle \, \, \text{with}\,\,i<j, \\[1ex] \m_{1,1}(\alpha_i)| 0 \rangle :=\frac{1}{2}\left( \q_1(\alpha_i)^2-\q_2(\alpha_i)\right) | 0 \rangle,
			\end{array}
		\end{equation*}
		with $i, j=1, \dots , 22$.
	\end{enumerate}
\end{thm}
\begin{rmk} \label{rmk delta and other line bundles in terms of Nakajima op}
	 By Definition \ref{defn Nakajima operators} with $k=2$, $n=0$, we have $\frac{1}{2}\q_2(1)|0\rangle=\delta$, where $\delta \in H^2(S^{[2]}, \Z)$ is such that $2\delta$ is the class of the exceptional divisor of the Hilbert--Chow morphism. Moreover, if $\alpha \in H^2(S, \Z)$, then by Definition \ref{defn Nakajima operators} with $n=k=1$ we obtain that $\alpha$ seen as an element of $H^2(S^{[2]}, \Z)$ is represented by $\q_1(1)\q_1(\alpha)|0\rangle$.
\end{rmk}
Consider the diagonal embedding $\tau_2: S \rightarrow S \times S$ of the K3 surface $S$. We denote the Gysin homomorphism followed by the Künneth isomorphism by $\tau_{2*}: H^*(S, \Z) \rightarrow H^*(S, \Z) \otimes H^*(S, \Z)$. We take the following basis $\{\alpha_1, \dots , \alpha_{22}\}$ for the lattice $H^2(S, \Z)$: let $\{\alpha_1, \dots , \alpha_8\}$ and $\{\alpha_9, \dots , \alpha_{16}\}$ be the basis of the two copies of $E_8(-1)$ whose Gram matrix is the second matrix of Example $\ref{eq Gram matrices}$ and $\{\alpha_{17}, \alpha_{18}\}$, $\{\alpha_{19}, \alpha_{20}\}$, $\{\alpha_{21}, \alpha_{22}\}$ be the basis of the three copies of $U$ whose Gram matrix is the first matrix in Example $\ref{eq Gram matrices}$. The following Lemma will be useful, see \cite[p.18]{boissiere2013smith}.
\begin{lem} \label{lem calcolo mu_ij}
	Let $X=S^{[2]}$ be the Hilbert square of a projective K3 surface $S$. Assume that $\{\alpha_1, \dots , \alpha_{22}\}$ is the basis of the lattice $H^2(S, \Z)$ constructed above. Then
	\begin{equation*}
		\tau_{2*}1=\sum_{i, j} \mu_{i, j} \alpha_i \otimes \alpha_j +1 \otimes x + x \otimes 1,
	\end{equation*}
	where the $\mu_{i,j}$'s are represented in Table $\ref{tabella mu}$ (we write down only the $\mu_{i,j}$'s which are non zero and such that $i \le j$):
	\begin{center}
		\begin{table}[H]
			\begin{tabular}{|c|c|c|c|c|}
				\hline
				$\mu_{1,1}=-4$ & $\mu_{1,2}=-7$ & $\mu_{1,3}=-10$ & $\mu_{1,4}=-8$ & $\mu_{1, 5}=-6$ \\ 
				\hline
				$\mu_{1, 6}=-4$ & $\mu_{1,7}=-2$ & $\mu_{1,8}=-5$ & $\mu_{2,2}=-14$ & $\mu_{2,3}=-20$ \\ 
				\hline
				$\mu_{2,4}=-16$ & $\mu_{2,5}=-12$ & $\mu_{2,6}=-8$ & $\mu_{2,7}=-4$ & $\mu_{2,8}=-10$ \\
				\hline
				$\mu_{3,3}=-30$ & $\mu_{3,4}=-24$ & $\mu_{3,5}=-18$ & $\mu_{3,6}=-12$ & $\mu_{3,7}=-6$ \\
				\hline
				$\mu_{3,8}=-15$ & $\mu_{4,4}=-20$ & $\mu_{4,5}=-15$ & $\mu_{4,6}=-10$ & $\mu_{4,7}=-5$ \\
				\hline
				$\mu_{4,8}=-12$ & $\mu_{5,5}=-12$ & $\mu_{5,6}=-8$ & $\mu_{5,7}=-4$ & $\mu_{5,8}=-9$ \\
				\hline
				$\mu_{6,6}=-6$ & $\mu_{6,7}=-3$ & $\mu_{6,8}=-6$ & $\mu_{7,7}=-2$ & $\mu_{7,8}=-3$ \\
				\hline
				$\mu_{8,8}=-8$ & $\mu_{9,9}=-4$ & $\mu_{9,10}=-7$ & $\mu_{9,11}=-10$ & $\mu_{9,12}=-8$ \\
				\hline
				$\mu_{9,13}=-6$ & $\mu_{9,14}=-4$ & $\mu_{9,15}=-2$ & $\mu_{9,16}=-5$ & $\mu_{10,10}=-14$ \\
				\hline
				$\mu_{10,11}=-20$ & $\mu_{10,12}=-16$ & $\mu_{10,13}=-12$ & $\mu_{10,14}=-8$ & $\mu_{10,15}=-4$ \\
				\hline
				$\mu_{10,16}=-10$ & $\mu_{11,11}=-30$ & $\mu_{11,12}=-24$ & $\mu_{11,13}=-18$ & $\mu_{11,14}=-12$ \\
				\hline
				$\mu_{11,15}=-6$ & $\mu_{11,16}=-15$ & $\mu_{12,12}=-20$ & $\mu_{12,13}=-15$ & $\mu_{12,14}=-10$ \\
				\hline
				$\mu_{12,15}=-5$ & $\mu_{12,16}=-12$ & $\mu_{13,13}=-12$ & $\mu_{13,14}=-8$ & $\mu_{13,15}=-4$ \\
				\hline
				$\mu_{13,16}=-9$ & $\mu_{14,14}=-6$ & $\mu_{14,15}=-3$ & $\mu_{14,16}=-6$ & $\mu_{15,15}=-2$ \\
				\hline
				$\mu_{15,16}=-3$ & $\mu_{16,16}=-8$ & $\mu_{17,18}=1$ & $\mu_{19,20}=1$ & $\mu_{21,22}=1$ \\
				\hline
			\end{tabular}
		\vspace{0.3cm}
			\caption{\label{tabella mu} The $\mu_{i,j}$'s.}
				\vspace{-1.3cm}
		\end{table}
	\end{center}
\end{lem}
\begin{proof}
	We can write 
	\begin{equation*} 
		\tau_{2*}1=\sum_{i, j} \mu_{i, j} \alpha_i \otimes \alpha_j +a (1 \otimes x) + b (x \otimes 1)
	\end{equation*}
	for some $\mu_{i, j} \in \Z$ such that $\mu_{i, j}=\mu_{j, i}$ with $i, j \in \{1, \dots , 22\}$, and some $a, b \in \Z$. Let $\langle \, \cdot \, , \, \cdot \, \rangle$ be the intersection pairing of $\mathbb{H}^S$ in (\ref{eq int pairing}).  Since $\tau_{2*}$ is the adjoint of the cup product, we have the relation $\langle \tau_{2*}1, \alpha_k \otimes \alpha_l \rangle = \int_S \alpha_k\alpha_l$, which gives, together with $\tau_{2*}1=\sum_{i, j} \mu_{i, j}\alpha_i\otimes\alpha_j+a(1\otimes x) + b(x \otimes 1)$, the following system:
	\begin{equation} \label{formula system}
		\displaystyle\sum_{i, j} \mu_{i, j} \int_S\alpha_i\alpha_k\int_S\alpha_j\alpha_l=\int_S\alpha_k\alpha_l.
	\end{equation}
	From (\ref{formula system}), with the help of a computer, we can compute the coefficients $\mu_{i, j}$. The solution of the system is given in Table $\ref{tabella mu}$ (we have written down only the $\mu_{i,j}$'s which are non zero and such that $i \le j$, since $\mu_{i,j}=\mu_{j,i}$). Similarly, from the relations
	\begin{equation*}
		\langle \tau_{2*}1, 1 \otimes x \rangle = \displaystyle\int_Sx=1, \qquad \langle \tau_{2*}1, x \otimes 1 \rangle =\displaystyle\int_Sx=1,
	\end{equation*}
	we obtain $a=b=1$.
\end{proof}
We now introduce the algebraic model developed by Lehn and Sorger in \cite{lehn2003cup}. Recall that a \emph{graded Frobenius algebra} of degree $d$ is a finite dimensional graded vector space $A=\bigoplus_{i=-d}^{d}A^i$ endowed with a graded commutative and associative multiplication $A \otimes A \rightarrow A$ of degree $d$, i.e., $\Deg(ab)=\Deg(a)+\Deg(b)+d$, and unit element $1$, necessarily of degree $-d$, together with a linear form $T: A \rightarrow \Q$ of degree $-d$ such that the induced symmetric bilinear form $\langle a, b \rangle:=T(ab)$ is non-degenerate and of degree $0$. Note that $d$ must be an even number since $1\cdot 1=1$. An example of graded Frobenius algebra is the shifted cohomology ring $H^*(X, \Q)[d]$ of a compact complex manifold $X$ of complex dimension $d$. From now on we consider $A:=H^*(S, \Q)[2]$, where $S$ is a projective K3 surface. The linear form $T:A \rightarrow \Q$ is $T(\alpha):=-\int_S \alpha$, and the induced bilinear form is $\langle \alpha, \beta \rangle=T(\alpha\beta)=-\int_S \alpha\beta$, which is the intersection pairing changed of sign. The graded Frobenius algebra structure of $A$ induces a structure of graded Frobenius algebra on $A^{\otimes n}$: as remarked in \cite{harvey2012characterizing}, since $S$ is a K3 surface, $A$ has only graded pieces of even degree, so the general construction of Lehn-Sorger simplifies. In this case, the multiplication induced on $A^{\otimes n}$ is
\begin{equation*}
	(a_1 \otimes \dots \otimes a_n) \cdot (b_1 \otimes \dots \otimes b_n)=(a_1b_1)\otimes \dots \otimes (a_nb_n),
\end{equation*}
and the linear form is
\begin{equation*}
	T: A^{\otimes n} \rightarrow \Q, \qquad a_1 \otimes \dots \otimes a_n \mapsto T(a_1)T(a_2)\dots T(a_n).
\end{equation*}
The symmetric group $S_n$ of order $n$ acts on $A^{\otimes n}$ as
\begin{equation*}
	\pi(a_1 \otimes \dots \otimes a_n):=a_{\pi^{-1}(1)} \otimes \dots \otimes a_{\pi^{-1}(n)}.
\end{equation*}
Given a partition $n=n_1+\dots + n_k$ with $n_1, \dots , n_k \in \Z_{>0}$, there is a generalised multiplication map $A^{\otimes n} \rightarrow A^{\otimes k}$ defined by
\begin{equation*}
	a_1 \otimes \dots \otimes a_n \mapsto (a_1\dots a_{n_1}) \otimes \dots \otimes (a_{n_1+\dots +n_{k-1}+1}\dots a_{n_1+\dots +n_k}).
\end{equation*}
Given a finite set $I \subset \{1, \dots , n\}$, let $A^{\otimes I}$ denote the tensor power with factors indexed by elements of $I$. Let $n$ be a fixed positive integer, $\pi \in S_n$ be a permutation and $\langle \pi \rangle \subset S_n$ be the subgroup generated by $\pi$. If $[n]:=\{1, \dots , n\}$, denote by $\langle \pi \rangle \setminus [n]$ the set of orbits. Then we set
\begin{equation*}
	A\{S_n\}:=\displaystyle\bigoplus_{\pi \in S_n} A^{\otimes \langle \pi \rangle \setminus [n]}\cdot \pi.
\end{equation*}
\begin{exmp}
	If $n=2$, we have $S_2=\{\text{id}, (1\,2)\}$. Moreover, $\langle \text{id} \rangle \setminus [2]=\{\{1\}, \{2\}\}$ and $\langle (1\,2)\rangle \setminus [2]=\{\{1, 2\}\}$, hence we obtain $A\{S_2\}=A^{\otimes 2} \text{id} \oplus A(1\,2)$. Similarly for $n=3$ we have
\begin{equation*}
	A\{S_3\}=A^{\otimes 3}\text{id}\oplus A^{\otimes 2}(1\,2)\oplus A^{\otimes 2}(1\,3)\oplus A^{\otimes 2}(2\,3)\oplus A(1\,2\,3)\oplus A(1\,3\,2).
\end{equation*}
\end{exmp}
Let $\sigma \in S_n$. There is a bijection
\begin{equation*}
	\sigma: \langle \pi \rangle \setminus [n] \rightarrow \langle \sigma \pi \sigma^{-1}\rangle \setminus [n], \qquad x \mapsto \sigma x
\end{equation*}
and an isomorphism
\begin{equation*}
	\tilde{\sigma}: A\{S_n\} \rightarrow A\{S_n\}, \qquad a\pi \mapsto \sigma(a)\sigma\pi\sigma^{-1}.
\end{equation*}
Thus we obtain an action of the symmetric group $S_n$ on $A\{S_n\}$. We denote by
\begin{equation*}
	A^{[n]}:=(A\{S_n\})^{S_n}
\end{equation*}
the subspace of invariants. We can now state the main theorem of \cite{lehn2003cup}.
\begin{thm}[Theorem 3.2 in \cite{lehn2003cup}] \label{thm lehn sorger isom}
	Let $S$ be a projective K3 surface. Then there is a canonical isomorphism of graded rings
	\begin{equation*}
		(H^*(S, \Q)[2])^{[n]} \xrightarrow{\sim} H^*(S^{[n]}, \Q)[2n].
	\end{equation*}
\end{thm}
The structure of graded Frobenius algebra of $H^*(S^{[n]}, \Q)[2n]$ is obtained by setting
\begin{equation*}
	T(a):=(-1)^n \displaystyle\int_{S^{[n]}}a \qquad \text{for all}\,\, a \in H^*(S^{[n]}, \Q).
\end{equation*}
We sketch the description of the isomorphism of Theorem \ref{thm lehn sorger isom}. Let $\mathcal{V}(A):=\text{Sym}^*(A \otimes t^{-1}\Q[t^{-1}])$, which is bigraded by degree and weight: an element $a \otimes t^{-m} \in A \otimes t^{-m}$ has degree $|a|$ and weight $m$. The component of $\mathcal{V}(A)$ of constant weight $n$ is the graded vector space
\begin{equation*}
	\mathcal{V}(A)_n \cong \displaystyle\bigoplus_{||\alpha||=n} \displaystyle\bigotimes_i \text{Sym}^{\alpha_i}A,
\end{equation*}
where the direct sum is taken over all the possible partitions $\alpha=(1^{\alpha_1}, 2^{\alpha_2}, \dots)$ of $n$, and $||\alpha||:=\alpha_1 \cdot 1+\alpha_2\cdot 2+ \dots$. We now fix $\pi \in S_n$. Let $f: \{1, 2, \dots , N\} \rightarrow \langle \pi \rangle \setminus [n]$ be an enumeration of the orbits of $\pi \in S_n$. We denote by $l_i$ the length of the $i$-th orbit, i.e., $l_i:=|f(i)|$. We define $\Phi^{\prime}: A^{\otimes N} \rightarrow \mathcal{V}(A)$ as
\begin{equation} \label{eq def Phi prime}
	a_1 \otimes \dots \otimes a_N \xmapsto{\Phi^{\prime}} \frac{1}{n!}(a_1\otimes t^{-l_1})\dots (a_N \otimes t^{-l_N}).
\end{equation}
Let $\Phi: \bigoplus_{n \ge 0} A\{S_n\} \rightarrow \mathcal{V}(A)$ be defined on each summand $A^{\otimes \langle \pi \rangle \setminus [n]}\cdot \pi$ by the composition
\begin{equation*}
	A^{\otimes \langle \pi \rangle \setminus [n]}\cdot \pi \xrightarrow{\tilde{f}^{-1}} A^{\otimes N} \xrightarrow{\Phi^{\prime}} \mathcal{V}(A),
\end{equation*}
where $\tilde{f}^{-1}$ denotes the identification of $A^{\otimes \langle \pi \rangle \setminus [n]} \cdot \pi$ with $A^{\otimes N}$ through the enumeration $f$. Then $\Phi$ induces an isomorphism of graded vector spaces $A^{[n]} \xrightarrow{\sim} \mathcal{V}(A)_n$ by \cite[Proposition 2.11]{lehn2003cup}. Moreover, there is an isomorphism of graded vector spaces $\Psi: \mathcal{V}(A) \rightarrow \bigoplus_{n \ge 0} H^*(S^{[n]}, \Q)[2n]$ given by
\begin{equation} \label{eq def Psi}
	(a_1t^{-n_1})\dots (a_st^{-n_s}) \xmapsto{\Psi} \q_{n_1}(a_1)\dots \q_{n_s}(a_s)|0\rangle,
\end{equation}
see \cite[Theorem 3.6]{lehn2003cup}. Combining (\ref{eq def Phi prime}) and (\ref{eq def Psi}) we can obtain a basis of $H^*(S^{[n]}, \Q)[2n]$ from a basis of $H^*(S, \Q)[2]$. In order to use correctly the isomorphism of Theorem \ref{thm lehn sorger isom}, if $A:=H^*(S, \Q)[2]$, we have to work with elements of $A\{S_n\}$ which are invariant for the action of $S_n$. The basis of $H^*(S^{[n]}, \Q)[2n]$ obtained is clearly also a basis of $H^*(S^{[n]}, \Q)$ with the standard grading. For instance, let $\{\alpha_1, \dots , \alpha_{22}\}$ be a basis of $H^2(S, \Q)$, denote by $1 \in H^0(S, \Z)$ the unit and by $x \in H^4(S, \Z)$ the class of a point: one can show that $\{\frac{1}{2}\q_2(1)|0\rangle, \q_1(1)\q_1(\alpha_i)|0\rangle\}_{i=1, \dots , 22}$ is a basis of $H^2(S^{[2]}, \Q)$ and $\{\frac{1}{2}\q_2(\alpha_i)|0\rangle, \q_1(1)\q_1(x)|0\rangle, \q_1(\alpha_i)\q_1(\alpha_j)|0\rangle\}_{1 \le i \le j \le 22}$ is a basis of $H^4(S^{[2]}, \Q)$.

We omit the definition of the product on $(H^*(S, \Q)[2])^{[n]}$ defined in \cite{lehn2003cup} to give the structure of ring. We state the following result which describes the cup product between elements in $H^2(S^{[2]}, \Z)$, see \cite[p.13]{boissiere2013smith}.
\begin{lem} \label{lem cup products Lehn-Sorger}
	Let $X=S^{[2]}$ be the Hilbert square of a projective K3 surface $S$. Let $\{\alpha_1, \dots , \alpha_{22}\}$ be the basis of the lattice $H^2(S, \Z)$ used in Lemma $\ref{lem calcolo mu_ij}$. Then the following equalities hold in $H^4(X, \Z)$.
	\begin{thmlist}
		\item For every $\alpha \in H^2(S, \Z)$ we have
		\begin{equation*}
			\frac{1}{2}\q_2(1)|0\rangle \cup \q_1(1)\q_1(\alpha)| 0 \rangle =\q_2(\alpha)|0\rangle.
		\end{equation*} \label{lem cup products Lehn-Sorger 1}
		\item For every $\alpha, \beta \in H^2(S, \Z)$ we have
		\begin{equation*}
			\q_1(1)\q_1(\alpha) | 0 \rangle \cup \q_1(1)\q_1(\beta)|0\rangle = \left( \int_S \alpha\beta \right) \q_1(1)\q_1(x) | 0 \rangle + \q_1(\alpha)\q_1(\beta)| 0 \rangle.
		\end{equation*} \label{lem cup products Lehn-Sorger 2}
		\item If $\mu_{i,j}$, with $i, j = 1, \dots , 22$, are the coefficients computed in Lemma $\ref{lem calcolo mu_ij}$, then
		\begin{equation*} 
			\frac{1}{2}\q_2(1)|0\rangle \cup \frac{1}{2}\q_2(1)|0\rangle = -\displaystyle\sum_{i<j} \mu_{i, j}\q_1(\alpha_i)\q_1(\alpha_j)|0\rangle - \frac{1}{2}\displaystyle\sum_i \mu_{i, i}\q_1(\alpha_i)^2|0 \rangle - \q_1(1)\q_1(x)|0\rangle.
		\end{equation*} \label{lem cup products Lehn-Sorger 3}
	\end{thmlist}
\end{lem}
\begin{rmk} \label{rmk errore delta^2}
	The result for the product $\frac{1}{2}\q_2(1)|0\rangle \cup \frac{1}{2}\q_2(1)|0\rangle$ given in \cite[p.18]{boissiere2013smith} is not correct: a change of sign is needed in the right-hand side. The map $\Delta_*$ in \cite[p.18]{boissiere2013smith} corresponds to $\tau_{2*}$ in our notation: in that article the cohomology ring taken is $H^*(S, \Q)$, without the shifting used in the Lehn--Sorger model, which gives the change of sign of the intersection pairing on $H^*(S, \Q)[2]$, as seen above. See also \cite[Remark 3.1]{harvey2012characterizing}.
\end{rmk}
	\section{Rational Hodge classes of type $(2,2)$ on Hilbert squares of K3 surfaces} \label{section rational Hodge}
	Let $S$ be a projective K3 surface. In this section we use Theorem \ref{thm lehn sorger isom} to compute a basis of the $\Q$-vector space $H^{2,2}(S^{[2]}, \Q)$ of rational Hodge classes of type $(2, 2)$ on $S^{[2]}$. 
	
	Recall that the rational cohomology groups $H^{2i}(S, \Q)$ and $H^{2j}(S^{[n]}, \Q)$, where $i, j \in \Z_{\ge 0}$, are Hodge structures of weight $2i$ and $2j$ respectively, and the shifted cohomology groups $H^{2i}(S, \Q)[2]$ and $H^{2j}(S^{[n]}, \Q)[2n]$ are Hodge structures of weight $2i-2$ and $2j-2n$ respectively, with the following Hodge decompositions:
	\begin{equation*}
			H^{2i}(S, \C)[2]=\displaystyle\bigoplus_{p+q=2i-2}H^{p, q}(S)[2], \qquad H^{2j}(S^{[n]}, \Q)[2n]=\displaystyle\bigoplus_{r+s=2j-2n}H^{r, s}(S^{[n]})[2n],
	\end{equation*}
	where $p, q \in \{-1, 0, \dots , i-1\}$ and $r, s \in \{-n, 1-n, \dots , j-n\}$, and
	\begin{equation*}
		H^{p, q}(S)[2] = H^{p+1, q+1}(S), \qquad H^{r, s}(S^{[n]})[2n] = H^{r+n, s+n}(S^{[n]}).
	\end{equation*}
For details on Hodge structures, see \cite[$\S 3$]{huybrechts2016lectures}. If $A:=H^*(S, \Q)[2]$, recall that $A^{[n]} \cong \bigoplus_{||\alpha||=n}\bigotimes_i \text{Sym}^{\alpha_i}A$. We denote by $(A^{[n]})^{2i}$ the component of degree $2i-2n$ of $A^{[n]}$. Then the Hodge structures $H^{2j}(S, \Q)[2]$ give rise to a Hodge structure on $(A^{[n]})^{2i}$. Since the weights of the Hodge structures considered depend only on the (shifted) cohomological degrees, we have that $(A^{[n]})^{2i}$ is a Hodge structure of weight $2i-2n$. Note that $H^{2i}(S^{[n]}, \Q)[2n]$ is a Hodge structure of weight $2i-2n$ and $(A^{[n]})^{2i} \cong H^{2i}(S^{[n]}, \Q)[2n]$ by Theorem \ref{thm lehn sorger isom}. Recall that, given two Hodge structures $V$ and $W$, a morphism of weight $k$ is a linear map $f: V \rightarrow W$ such that its $\C$-linear extension satisfy $f(V^{p, q}) \subseteq W^{p+k, q+k}$. The pullback and the Gysin homomorphism induced by a morphism $f:X \rightarrow Y$ between two complex projective manifolds are morphisms of Hodge structures of weight respectively $0$ and $r=\Dim_{\C}(Y)-\Dim_{\C}(X)$, see \cite[$\S 7.3.2$]{voisin2002theorie}. Using this property, Definition \ref{defn Nakajima operators} and the isomorphisms (\ref{eq def Phi prime}) and (\ref{eq def Psi}) one can show the following result, which is well known to experts and implicitly given in \cite{lehn2003cup}.
\begin{thm} \label{thm LS isom compatible with Hodge dec}
	Let $S$ be a projective K3 surface. Let $A:=H^*(S, \Q)[2]$ and consider $(A^{[n]})^{2i} \cong H^{2i}(S^{[n]}, \Q)[2n]$, the isomorphism induced by Theorem $\ref{thm lehn sorger isom}$ on the components of (shifted) cohomological degree $2i-2n$. Take on $(A^{[n]})^{2i}$ the Hodge structure of weight $2i-2n$ described above and on $H^{2i}(S^{[n]}, \Q)[2n]$ the Hodge structure of weight $2i-2n$ induced by shifted cohomology. Then
	\begin{equation*}
		(A^{[n]})^{2i} \xrightarrow{\sim} H^{2i}(S^{[n]}, \Q)[2n]
	\end{equation*}
	is an isomorphism of Hodge structures of weight $0$. 
\end{thm}
Let $S$ be a K3 surface and $T(S)$ be its transcendental lattice, which is a Hodge structure of weight $2$. Denote by $E_S:=\text{Hom}_0(T(S)_{\Q}, T(S)_{\Q})$ the algebra of endomorphisms on $T(S)_{\Q}$ of weight $0$. Recall that a K3 surface $S$ is \emph{general} if $E_S\cong \Q$. We can now give a basis of the $\Q$-vector space $H^{2,2}(S^{[2]}, \Q)$ of rational Hodge classes of type $(2, 2)$ on the Hilbert square of a general projective K3 surface.
\begin{thm} \label{thm dim H2,2(X, Q) per K3 qualsiasi}
	Let $S$ be a general projective K3 surface and let $\{b_1, \dots , b_r\}$ be a basis of $\Pic(S)$. Then:
	\begin{enumerate}[label=(\roman*)]
		\item $\Dim(H^{2,2}(S^{[2]}, \Q))=\frac{(r+1)r}{2}+r+2$.
		\item A basis of $H^{2,2}(S^{[2]}, \Q)$ is given by the following elements:
		\begin{itemize}
			\item $\frac{1}{2}\q_2(b_i)|0\rangle$ for $i=1, \dots , r$.
			\item $\q_1(1)\q_1(x)|0\rangle$, where $1 \in H^0(S, \Q)$ is the unit and $x \in H^4(S, \Q)$ is the class of a point.
			\item $\frac{1}{2}\q_1(b_i)^2|0\rangle$ for $i=1, \dots , r$.
			\item $\q_1(b_i)\q_1(b_j)|0\rangle$ for $1 \le i < j \le r$.
			\item $-\displaystyle\sum_{i<j} \mu_{i,j}\q_1(\alpha_i)\q_1(\alpha_j)|0\rangle-\frac{1}{2}\displaystyle\sum_i\mu_{i,i}\q_1(\alpha_i)^2|0\rangle-\q_1(1)\q_1(x)|0\rangle$, 
			\newline
			where $\{\alpha_1, \dots , \alpha_{22}\}$ is the basis of $H^2(S, \Z)$ used in Lemma $\ref{lem calcolo mu_ij}$ and the $\mu_{i,j}$'s are given in Table $\ref{tabella mu}$.
		\end{itemize}
	\end{enumerate}
\end{thm}
\begin{proof}
	By Theorem \ref{thm lehn sorger isom} and Theorem \ref{thm LS isom compatible with Hodge dec} the following is an isomorphism of Hodge structures of weight $0$:
	\begin{equation} \label{isoojojm}
		H^4(S^{[2]}, \Q) \cong H^2(S, \Q) \oplus \left( H^0(S, \Q) \otimes H^4(S, \Q) \right) \oplus \Sym^2(H^2(S, \Q)),
	\end{equation}
	we omit the shiftings of the cohomology groups. The Hodge classes of $S^{[2]}$ of bidegree $(2, 2)$ have bidegree $(0, 0)$ in the shifted cohomology, so we look for the components of (shifted) bidegree $(0, 0)$ in the right-hand side of $(\ref{isoojojm})$. If $\NS(S)$ is the Néron--Severi group of $S$ and $T(S)=(\NS(S))^{\perp}$ is the transcendental lattice, $H^2(S, \Q)$ can be decomposed as
	\begin{equation} \label{eq dec H2(S) con NS e T(S)}
		H^2(S, \Q) \cong \NS(S)_{\Q} \oplus T(S)_{\Q},
	\end{equation}
	where $\NS(S)_{\Q}=\NS(S) \otimes \Q$ and $T(S)_{\Q}=T(S) \otimes \Q$.
	\newline
	The first summand of $(\ref{isoojojm})$ has the $\Q$-vector space $\NS(S)_{\Q}[2]$ as component of bidegree $(0, 0)$. Since $\NS(S) \cong \Pic(S)$ and $\rk(\Pic(S))=r$ by assumption, the component of bidegree $(0, 0)$ of $H^2(S, \Q)[2]$ has dimension $r$. By (\ref{eq def Phi prime}) and (\ref{eq def Psi}) we have $b_i \xmapsto{\Psi \circ \Phi^{\prime}} \frac{1}{2}\q_2(b_i)|0\rangle$, so $\frac{1}{2}\q_2(b_i)|0\rangle$ is in a basis of $H^{2,2}(S^{[2]}, \Q)$ for $i=1, \dots , r$.
	\newline
	The second summand of $(\ref{isoojojm})$ is $H^0(S, \Q)[2] \otimes H^4(S, \Q)[2]$, which is a vector space over $\Q$ of dimension $1$. This is generated by $1 \otimes x$, which is an element of bidegree $(0, 0)$. Since $1\otimes x + x \otimes 1 \xmapsto{\Psi \circ \Phi^{\prime}} \q_1(1)\q_1(x)|0\rangle$, the element $\q_1(1)\q_1(x)|0\rangle$ is in a basis of $H^{2,2}(S^{[2]}, \Q)$.
	\newline
	Consider $\Sym^2(H^2(S, \Q)[2])$, the third summand of $(\ref{isoojojm})$. Using $(\ref{eq dec H2(S) con NS e T(S)})$, we can decompose it as
	\begin{equation} \label{eq decomposition Sym2(H2)}
			\Sym^2(H^2(S, \Q)[2]) \cong \Sym^2(\Pic(S)_{\Q}[2]) \oplus \Sym^2(T(S)_{\Q}[2]) \oplus \left(\Pic(S)_{\Q}[2] \otimes T(S)_{\Q}[2]\right).
	\end{equation}
	By assumption $\rk(\Pic(S))=r$, so $\Sym^2(\Pic(S)_{\Q}[2])$, whose elements have all bidegree $(0, 0)$, has dimension $\frac{(r+1)r}{2}$ as $\Q$-vector space. By (\ref{eq def Phi prime}) and (\ref{eq def Psi}) we have:
	\begin{equation*}
			b_i \otimes b_i \xmapsto{\Psi \circ \Phi^{\prime}} \frac{1}{2}\q_1(b_i)^2|0\rangle, \qquad b_i \otimes b_j + b_j \otimes b_i \xmapsto{\Psi \circ \Phi^{\prime}} \q_1(b_i)\q_1(b_j)|0\rangle
	\end{equation*}
	for $i, j \in \{1, \dots, r\}$ and $i<j$. Then the elements $\frac{1}{2}\q_1(b_i)^2|0\rangle$, for $i=1, \dots , r$, and $\q_1(b_i)\q_1(b_j)|0\rangle$, for $1 \le i < j \le r$, are in a basis of $H^{2,2}(S^{[2]}, \Q)$. Note that $\Pic(S)_{\Q}[2]\otimes T(S)_{\Q}[2]$ does not contain any element of bidegree $(0, 0)$. It remains to determine the elements of bidegree $(0, 0)$ in $\Sym^2(T(S)_{\Q}[2])$, i.e.,
	\begin{equation*}
		(\Sym^2(T(S)_{\Q}[2]))^{0, 0} \cap (\Sym^2(T(S)_{\Q}[2])).
	\end{equation*}
	Consider $T(S)_{\Q}\otimes T(S)_{\Q}$ with the standard grading. Then $T(S)_{\Q}$ is the minimal sub-Hodge structure of $H^2(S, \Q)$ with $H^{2,0}(S) = T(S)_{\Q}^{2,0}$, see \cite[Definition 3.2.5, Lemma 3.3.1]{huybrechts2016lectures}, in particular $T(S)_{\Q}$ is a Hodge structure of weight $2$. By \cite[Example 3.1.3, (iii)]{huybrechts2016lectures} the dual $T(S)_{\Q}^*=\text{Hom}(T(S)_{\Q}, \Q)$ is a Hodge structure of weight $-2$, and there is an isomorphism of Hodge structures of weight $-2$ from $T(S)_{\Q}$ to $T(S)_{\Q}^*$. This implies that
	\begin{equation*}
		(T(S)_{\Q} \otimes T(S)_{\Q})^{2,2} \xrightarrow{\sim} (T(S)_{\Q}^* \otimes T(S)_{\Q})^{0,0},
	\end{equation*}
	and by \cite[Example 3.1.4, (iv)]{huybrechts2016lectures} we have
	\begin{equation*}
		E_S=\text{Hom}_0(T(S)_{\Q}, T(S)_{\Q}) \cong (T(S)_{\Q}^* \otimes T(S)_{\Q}) \cap (T(S)_{\Q}^* \otimes T(S)_{\Q})^{0, 0},
	\end{equation*}
	where $\text{Hom}_0(T(S)_{\Q}, T(S)_{\Q})$ denotes the space of Hodge endomorphisms on $T(S)_{\Q}$ of weight $0$. Since by assumption $S$ is a general K3 surface we have $\text{Hom}_0(T(S)_{\Q}, T(S)_{\Q}) \cong \Q \cdot \text{id}$. Passing to the shifted cohomology groups, this implies that the $\Q$-vector space $(\Sym^2(T(S)_{\Q}[2]))^{0, 0} \cap (\Sym^2(T(S)_{\Q}[2]))$ has dimension $1$. We now describe the element induced by its generator on $H^4(S^{[2]}, \Q)$. Let $\{\beta_{r+1}, \dots , \beta_{22}\}$ be an orthogonal basis of $T(S)_{\Q}$ with respect to the intersection form, and let $\{\beta_{r+1}^{\vee}, \dots , \beta_{22}^{\vee}\}$ be the basis of $T(S)^*_{\Q}$ given by $\beta_i^{\vee}:=(\beta_i, \,\cdot\, ) \in \text{Hom}(T(S)_{\Q}, \Q) \cong T(S)_{\Q}^*$ for $i \in \{r+1, \dots, 22\}$. Then
	\begin{equation*}
		\id=\displaystyle\sum_{i=r+1}^{22}\frac{1}{(\beta_i, \beta_i)}\beta_i^{\vee} \otimes \beta_i \in T(S)_{\Q}^* \otimes T(S)_{\Q}
	\end{equation*}
	since for every $k \in \{r+1, \dots , 22\}$ we have
	\begin{equation*}
			\left( \displaystyle\sum_{i=r+1}^{22} \frac{1}{(\beta_i, \beta_i)}\beta_i^{\vee}\otimes \beta_i \right) (\beta_k) = \displaystyle\sum_{i=r+1}^{22} \frac{1}{(\beta_i, \beta_i)}(\beta_i, \beta_k)\cdot \beta_i = \beta_k.
	\end{equation*}
	Note that $(\beta_i, \beta_i) \neq 0$ since $\{\beta_{r+1}, \dots , \beta_{22}\}$ is an orthogonal basis and the intersection form on $H^2(S, \Q)$ is non-degenerate. Since $T(S)_{\Q} \cong T(S)_{\Q}^*$ by the map $\beta_i \mapsto \beta_i^{\vee}$, the identity, seen as element in $T(S)_{\Q} \otimes T(S)_{\Q}$, is
	\begin{equation} \label{eq id}
		\id = \displaystyle\sum_{i=r+1}^{22} \frac{1}{(\beta_i, \beta_i)} \beta_i \otimes \beta_i \in T(S)_{\Q} \otimes T(S)_{\Q}.
	\end{equation}
	We see that (\ref{eq id}) is invariant for the action of the symmetric group $S_2$ on $A\{S_2\}$, where $A:=H^*(S, \Q)[2]$, so we obtain the following element of $H^4(S^{[2]}, \Q)[4]$:
	\begin{equation} \label{smoooch}
		\id \xmapsto{\Psi \circ \Phi^{\prime}} \frac{1}{2}\displaystyle\sum_{i=r+1}^{22} \frac{1}{(\beta_i, \beta_i)} \q_1(\beta_i)^2|0\rangle.
	\end{equation}
	Hence the last element of the basis of $H^{2,2}(S^{[2]}, \Q)$ is $(\ref{smoooch})$. After some tedious computations, one can show that the element in (\ref{smoooch}) can be substituted in the basis of $H^{2,2}(S^{[2]}, \Q)$ obtained by the following element:
	\begin{equation*}
		\delta^2=-\displaystyle\sum_{i<j}\mu_{i,j}\q_1(\alpha_i)\q_1(\alpha_j)|0\rangle-\frac{1}{2}\displaystyle\sum_i\mu_{i,i}\q_1(\alpha_i)^2|0\rangle-\q_1(1)\q_1(x)|0\rangle.
	\end{equation*}
	We omit this part of the proof. We conclude that $\Dim(H^{2,2}(S^{[2]}, \Q))=\frac{(r+1)r}{2}+r+2$ and a basis of the $\Q$-vector space $H^{2,2}(S^{[2]}, \Q)$ is the one presented in the statement of the theorem.
\end{proof}
\begin{rmk} \label{rmk general K3}
	Note that the assumption that $S$ is a \emph{general} K3 surface, i.e., $E_S \cong \Q$, is necessary in the statement of Theorem \ref{thm dim H2,2(X, Q) per K3 qualsiasi}. If $S$ is general, as seen in the proof, the component $\text{Sym}^2(T(S)_{\Q}[2])$ in $(\ref{eq decomposition Sym2(H2)})$ gives only the element (\ref{eq id}) in a basis of $H^{2,2}(S^{[2]}, \Q)$ (we have substituted it with $\delta^2$ in the end of the proof). If $S$ is not general, $E_S$ is bigger than $\Q$, hence the component $\text{Sym}^2(T(S)_{\Q}[2])$ in $(\ref{eq decomposition Sym2(H2)})$ gives not only (\ref{eq id}), but also other elements in a basis of $H^{2,2}(S^{[2]}, \Q)$.
\end{rmk}
One can combine Theorem \ref{thm dim H2,2(X, Q) per K3 qualsiasi} with Theorem \ref{thm Qin--Wang} to find a basis, in terms of Nakajima operators, of the lattice $H^{2,2}(S^{[2]}, \Z)$ of integral Hodge classes of type $(2, 2)$ of a \emph{generic} K3 surface $S$, which is by definition a general projective K3 surface with Picard group of rank $1$. The proof of this result is very similar to the one of Theorem \ref{thm H^2,2(X, Z) generico}, which generalises it to the Hilbert square of \emph{any} general projective K3 surface. For this reason, we only sketch the proof: details can be found in the author's PhD thesis, see \cite[Theorem 3.3.17]{novario2021ths}. While Theorem \ref{thm H^2,2(X, Z) generico}, as we will see, needs Theorem \ref{thm 2/5q con op Nakajima} to be proven, Theorem\,\,\ref{thm H^22(X, Z) prima versione} can be shown without using it.
\begin{thm} \label{thm H^22(X, Z) prima versione}
	Let $S$ be a generic K3 surface and $h \in \Pic(S)$ be the ample generator of $\Pic(S)$. Then a basis of the lattice $H^{2,2}(S^{[2]}, \Z)$ of integral Hodge classes of type $(2, 2)$ is given by the following elements:
	\begin{equation} \label{eq basis H^2,2Z very general S}
		\begin{array}{c}
			\q_2(h)|0\rangle, \qquad \q_1(1)\q_1(x)|0\rangle, \qquad \frac{1}{2}\left( \q_1(h)^2-\q_2(h)\right) | 0 \rangle, \\[1ex]
			-\displaystyle\sum_{i<j}\mu_{i,j}\q_1(\alpha_i)\q_1(\alpha_j)|0\rangle-\frac{1}{2}\displaystyle\sum_i \mu_{i,i}\q_1(\alpha_i)^2|0\rangle-\q_1(1)\q_1(x)|0\rangle,
		\end{array}
	\end{equation}
	where $x \in H^4(S, \Z)$ is the class of a point, $1 \in H^0(S, \Z)$ is the unit, $\{\alpha_1, \dots, \alpha_{22}\}$ is the basis of $H^2(S, \Z)$ used in Lemma $\ref{lem calcolo mu_ij}$ and the $\mu_{i,j}'s$ are the integers given by Table $\ref{tabella mu}$.
\end{thm}
\begin{proof}[Sketch of the proof]
	After a slight modification of the basis given in Theorem \ref{thm dim H2,2(X, Q) per K3 qualsiasi}, we see that (\ref{eq basis H^2,2Z very general S}) is a basis of $H^{2,2}(S^{[2]}, \Q)$. In order to prove the theorem, we give a basis of the lattice $H^4(S^{[2]}, \Z)$ which contains the elements in (\ref{eq basis H^2,2Z very general S}). The Picard group $\Pic(S) \cong \Z h$ can be primitively embedded in a unique way up to isometries in $H^2(S, \Z)$, see \cite[Theorem 14.1.12]{huybrechts2016lectures}: we can identify $h$ with $\alpha_{17}+t\alpha_{18}$. After some computations, using Theorem \ref{thm Qin--Wang} and the fact that $\mu_{21, 22}=1$ by Table \ref{tabella mu} (compare with the proof of Theorem \ref{thm H^2,2(X, Z) generico}), it is possible to show that the following is a basis of $H^4(S^{[2]}, \Z)$:
	\begin{equation} \label{eq basis proof H4}
		\begin{array}{l}
			\q_1(1)\q_1(x)|0\rangle, \\[1ex]
			\q_2(\beta_i)|0\rangle \,\,\text{for}\,\,i=1, \dots , 22, \\[1ex]
			\q_1(\beta_i)\q_1(\beta_j)|0\rangle \,\,\text{for}\,\,1\le i<j\le 22 \,\,\,\text{and} \,\, (i,j) \neq (21, 22),\\[1ex]
			\frac{1}{2}\left( \q_1(\beta_i)^2-\q_2(\beta_i)\right)|0\rangle \,\,\text{for}\,\,i=1, \dots , 22, \\[1ex]
			-\displaystyle\sum_{i<j}\mu_{i,j}\q_1(\alpha_i)\q_1(\alpha_j)|0\rangle -\frac{1}{2}\displaystyle\sum_i \mu_{i,i}\q_1(\alpha_i)^2|0\rangle -\q_1(1)\q_1(x)|0\rangle,
		\end{array}
	\end{equation}
where $\beta_i:=\alpha_i$ for $i \neq 17$ and $\beta_{17}:=\alpha_{17}+t\alpha_{18}$. Since the elements in (\ref{eq basis H^2,2Z very general S}) are contained in (\ref{eq basis proof H4}) and $\beta_{17}=h$, we conclude that (\ref{eq basis H^2,2Z very general S}) is a basis of $H^{2,2}(S^{[2]}, \Z)$.
\end{proof}
	\section{Second Chern class of the Hilbert square of a K3 surface} \label{section second Chern}
	Let $S$ be a projective K3 surface and $X:=S^{[2]}$ be its Hilbert square. In this section we look for a representation of the second Chern class $c_2(X)$ of $X$ in terms of Nakajima operators: this will be fundamental to find a basis of the lattice $H^{2,2}(X, \Z)$ of integral Hodge classes of type $(2, 2)$ on the Hilbert square of a general projective K3 surface. Let $q_X^{\vee} \in H^{2,2}(X, \Q)$ be the dual of the BBF-form. As always we denote by $\delta \in \Pic(X)$ the class such that $2\delta$ is the class of the exceptional divisor of the Hilbert--Chow morphism. First of all, the following proposition gives a basis of $H^{2,2}(X, \Q)$ which does not depend on Nakajima operators, where $X$ is the Hilbert square of a generic K3 surface: we omit the proof.
	\begin{prop} \label{prop h^2 delta^2 hdelta q lin ind}
		Let $X=S^{[2]}$ be the Hilbert square of a generic K3 surface and let $h \in \Pic(X)$ be the class induced by the ample generator of $\Pic(S)$. Then $\{h^2, h\delta, \delta^2, \frac{2}{5}q_X^{\vee} \}$ is a basis of the $\Q$-vector space $H^{2,2}(X, \Q)$.
	\end{prop}
Theorem \ref{thm H^22(X, Z) prima versione} gives us a basis of the lattice $H^{2,2}(S^{[2]}, \Z)$, where $S$ is a generic K3 surface, in terms of Nakajima operators. It is natural to wonder how to describe this basis without Nakajima operators, in particular we want to obtain a basis of $H^{2,2}(S^{[2]}, \Z)$ in terms of some rational linear combination of $h^2, h\delta, \delta^2, \frac{2}{5}q_X^{\vee} \in H^{2,2}(S^{[2]}, \Q)$ of Proposition\,\,\ref{prop h^2 delta^2 hdelta q lin ind}. Note that using Lemma \ref{lem cup products Lehn-Sorger} we can represent $h^2$, $h\delta$ and $\delta^2$ in terms of Nakajima operators, while we do not know how to write $\q_1(1)\q_1(x)|0\rangle$ of Theorem \ref{thm H^22(X, Z) prima versione} in terms of $h^2, h\delta, \delta^2, \frac{2}{5}q_X^{\vee}$. Hence we need to express $\frac{2}{5}q_X^{\vee} \in H^{2,2}(X, \Z)$ in terms of Nakajima operators. Recall that $\frac{6}{5}q_X^{\vee}=c_2(X)$ by Proposition \ref{prop O'Grady c_2(X)}, so we look for a representation of $c_2(X) \in H^{2, 2}(X, \Z)$ in terms of Nakajima operators. The tool that we use is the EGL formula of Proposition \ref{prop EGL}

Let $S$ be a projective K3 surface. We denote by $\mathcal{T}_2:=\mathcal{T}_{S^{[2]}}$ the tangent bundle of $S^{[2]}$. From now on, we will denote by $\,\cdot\,$ the cup product. We define the following operator on the cohomology ring $H^*(S^{[2]}, \Q)$:  
\begin{equation*}
	\ch(\mathcal{T}_2): H^*(S^{[2]}, \Q) \rightarrow H^*(S^{[2]}, \Q), \qquad x \mapsto \ch(\mathcal{T}_2) \cdot x.
\end{equation*}
By the general construction of $S^{[n, n+k]}$ seen in Section $\ref{section Nakajima}$, if $\Delta \subset S \times S$ is the diagonal, we have $S^{[1, 2]} \cong \text{Bl}_{\Delta}(S^2)$. Diagram $(\ref{diagram EGL})$ for $n=1$ gives
\begin{equation} \label{diagram EGL n=1}
	\begin{tikzcd}
		& S^{[1, 2]} \arrow[bend right=10,swap]{ddl}{\varphi}   \arrow[bend left=10]{ddr}{\rho} \arrow[d, "\sigma"]  \arrow[r, "\psi"] & S^{[2]} \\ 
		& S \times S \arrow[dl, "p"] \arrow{dr}[swap]{q} & \\
		S & \Delta \arrow[u, hook, "\iota"] & S.
	\end{tikzcd}	
\end{equation}
Note that the morphisms $\varphi, \rho$ and $\psi$ appearing in diagram $(\ref{diagram EGL n=1})$ correspond to the morphisms $\varphi, \rho$ and $\psi$ of diagram $(\ref{diagram Nakajima projections})$, with $n=k=1$, precomposed with the inclusion of $S^{[1, 2]}$ in $S \times S \times S^{[2]}$. With this notation, the definition of the Nakajima operator $\q_1$ is the same of $(\ref{eq defn Nakajima op})$ without the component $PD^{-1}[S^{[1, 2]}]$, i.e., for $\alpha, x \in H^*(S)$ we have
\begin{equation*}
	\q_1(\alpha)(x)= \psi_* \left( \varphi^*(x) \cdot \rho^*(\alpha) \right).
\end{equation*}
By properties of cup and cap product, the latter denoted by $\cap$, we have $\varphi^*(x) \cdot \rho^*(\alpha)=PD^{-1}\left( [S^{[1, 2]}] \cap \varphi^*(x) \cdot \rho^*(\alpha)\right)$ in $H^*(S^{[1, 2]})$ for every $\alpha, x \in H^*(S)$. Then we get
\begin{equation*} 
	\begin{array}{lll}
		\ch(\mathcal{T}_2)\cdot \q_1(\alpha)(x) & = & \ch(\mathcal{T}_2) \cdot \psi_* \left( \varphi^*(x) \cdot \rho^*(\alpha) \right) \\[1ex]
		& = & \ch(\mathcal{T}_2) \cdot PD^{-1} \psi_*\left( [S^{[1, 2]}] \cap \varphi^*(x) \cdot \rho^*(\alpha) \right) \\[1ex]
		& = & PD^{-1}\psi_*\left( [S^{[1, 2]}] \cap \psi^*(\ch(\mathcal{T}_2))\cdot \varphi^*(x) \cdot \rho^*(\alpha)\right) \\[1ex]
		& = & PD^{-1}\psi_*\left( [S^{[1, 2]}] \cap \ch(\psi^!\mathcal{T}_2)\cdot \varphi^*(x) \cdot \rho^*(\alpha)\right),
	\end{array}
\end{equation*}
where $\psi_*$ in the first equality is the Gysin homomorphism, while in the other equalities is the pushforward in homology, and the third equality comes from the projection formula. Applying Proposition $\ref{prop EGL}$ with $n=1$ and $\omega_S$ trivial we get:
\begin{equation} \label{eq Chern character T2 Nakajima R1 R6}
	\begin{array}{lll}
		\ch(\mathcal{T}_2) \cdot \q_1(\alpha)(x) & = & PD^{-1}\psi_*\left( [S^{[1, 2]}] \cap \varphi^*(\ch(\mathcal{T}_S)\cdot x)\cdot \rho^*(\alpha) \right) \\[1ex]
		& & + PD^{-1}\psi_*\left( [S^{[1, 2]}] \cap \ch(\mathcal{L})\cdot \varphi^*(x) \cdot \rho^*(\alpha)\right) \\[1ex]
		& & -PD^{-1}\psi_* \left( [S^{[1, 2]}] \cap \ch(\mathcal{L})\cdot \sigma^*(\ch(\Ol_{\Delta}^{\vee})) \cdot \varphi^*(x) \cdot \rho^*(\alpha) \right) \\[1ex]
		& & +PD^{-1}\psi_*\left( [S^{[1, 2]}] \cap \ch(\mathcal{L^{\vee}})\cdot \varphi^*(x) \cdot \rho^*(\alpha)\right) \\[1ex]
		& & -PD^{-1}\psi_*\left( [S^{[1, 2]}] \cap \ch(\mathcal{L}^{\vee}) \cdot \sigma^*(\ch(\Ol_{\Delta})) \cdot \varphi^*(x) \cdot \rho^*(\alpha)\right) \\[1ex]
		& & -PD^{-1}\psi_*\left( [S^{[1, 2]}]\cap \varphi^*(x) \cdot \rho^*(\ch(2\Ol_S-\mathcal{T}_S))\right),
	\end{array}
\end{equation}
where $\mathcal{L}:=\Ol_{S^{[1, 2]}}(-N)$ and $N$ is the exceptional divisor of the blowing up $\sigma: \text{Bl}_{\Delta}(S^2) \rightarrow S^2$.
If we set $x:=\q_1(1)|0 \rangle$ and $\alpha:=1$, we can use formula (\ref{eq Chern character T2 Nakajima R1 R6}) to compute $c_2(S^{[2]})$ in terms of Nakajima operators. Recall that here the dual is defined by (\ref{eq defn dual}), while we denote by $F^*:=\text{Hom}(F, \Ol_X)$ the classical dual of a coherent sheaf $F$. We now study the duals appearing in the right-hand side of $(\ref{eq Chern character T2 Nakajima R1 R6})$.
\begin{lem} \label{lem duali formula EGL}
	Keep notation as above. Then:
	\begin{enumerate}[label=(\roman*)]
		\item The dual $\mathcal{L}^{\vee}$ is isomorphic to the dual $\mathcal{L}^*=\textnormal{Hom}(\mathcal{L}, \Ol_{S^{[1, 2]}})$.
		\item $\Ol_{\Delta}^{\vee}=\Ol_{\Delta}$.
	\end{enumerate}
\end{lem}
\begin{proof}
	\begin{enumerate}[label=(\roman*)]
		\item By \cite[Proposition III.6.7]{hartshorne2013algebraic} we have $\Ext^{\,i}(\mathcal{L}, \Ol_{S^{[1, 2]}}) \cong \Ext^{\,i}(\Ol_{S^{[1, 2]}}, \Ol_{S^{[1, 2]}}) \otimes \mathcal{L}^*$. Moreover, by \cite[Proposition III.6.3]{hartshorne2013algebraic} we have that $\Ext^{\,i}(\Ol_{S^{[1, 2]}}, \Ol_{S^{[1, 2]}})$ is $0$ for $i>0$ and it is equal to $\Ol_{S^{[1, 2]}}$ if $i=0$. We conclude that $\sum_i(-1)^i\Ext^{\,i}(\mathcal{L}, \Ol_{S^{[1, 2]}}) \cong \mathcal{L}^*$.
		\item We apply \cite[Lemma 1]{schnellserre} with $X=S \times S$, $Z=\Delta$ and $\mathcal{L}=\Ol_{S \times S}$. We have $\Ext^{\,i}(\Ol_{\Delta}, \Ol_{S \times S})=0$ if $i \neq 2$ and $\Ext^{\,2}(\Ol_{\Delta}, \Ol_{S \times S})=\text{det}\,N_{\Delta| S \times S}$. Moreover, $N_{\Delta|S \times S}=\mathcal{T}_S$ and $\text{det}\,\mathcal{T}_S \cong \Ol_S$ since $S$ is a K3 surface. We conclude that $\sum_i(-1)^1\Ext^{\,i}(\Ol_{\Delta}, \Ol_{S \times S})=\Ol_{\Delta}$, as we wanted.
	\end{enumerate}
\end{proof}
We recall the computation of the Chern character of $\Ol_{\Delta}:=i_*\Ol_{\Delta}$, where $i: \Delta \hookrightarrow S \times S$ is the inclusion.
\begin{lem} \label{lem ch(O_Delta)}
	Let $S$ be a K3 surface and $\Delta \subset S \times S$ be the diagonal. Denote by $[\Delta] \in H^4(S \times S, \Z)$ the fundamental cohomological class of $\Delta$ in $S\times S$. Then $\ch(\Ol_{\Delta})=[\Delta]-2y$, where $y \in H^8(S \times S, \Q)$ is the class of a point in $S \times S$.
\end{lem}
\begin{proof}
	Let $i: \Delta \hookrightarrow S \times S$ be the inclusion. By Grothendieck--Riemann--Roch Theorem, see \cite[Theorem 1]{atiyah1959riemann} and \cite{hirzebruch1962riemann}, we have
	\begin{equation*}
		\begin{array}{lll}
			\ch(\Ol_{\Delta}) & = & i_*\left(\ch(\Ol_{\Delta}) \cdot \td(S)\right)\cdot \td(S \times S)^{-1} \\
			& = & i_*\left(\td(S) \cdot i^*\td(S \times S)^{-1}\right) \\
			& = & i_*\left(\td(S)^{-1}\right),
		\end{array}
	\end{equation*}
	where $i_*$ is the Gysin homomorphism and $\Ol_{\Delta}$ in the left-hand side is $i_*\Ol_{\Delta}$. Let $x \in H^4(S, \Q)$ be the class of a point of $S$. The Todd class of a K3 surface is $\text{td}(S)=1+2x$, hence we obtain $\ch(\Ol_{\Delta})=[\Delta]-2y$.
\end{proof}
Consider formula $(\ref{eq Chern character T2 Nakajima R1 R6})$. We introduce the following notation: 
\begin{equation*}
	\begin{array}{l}
		L1:=\ch(\mathcal{T}_2)\cdot \q_1(\alpha)(x), \\[0.8ex]
		R1:= PD^{-1}\psi_{*}\left([S^{[1, 2]}]\cap \varphi^* (\ch(\mathcal{T}_S)\cdot x)\cdot \rho^*(\alpha)\right), \\[0.8ex]
		R2:= PD^{-1}\psi_{*}\left([S^{[1, 2]}]\cap \ch(\mathcal{L})\cdot \varphi^* (x) \cdot \rho^*(\alpha)\right), \\[0.8ex]
		R3:= PD^{-1}\psi_{*}\left([S^{[1, 2]}] \cap \ch(\mathcal{L}) \cdot \sigma^*(\ch(\Ol^{\vee}_{\Delta}))\cdot \varphi^*(x)\cdot \rho^*(\alpha)\right), \\[0.8ex]
		R4:= PD^{-1}\psi_{*}\left([S^{[1, 2]}] \cap \ch(\mathcal{L}^{\vee})\cdot \varphi^*(x)\cdot \rho^*(\alpha)\right), \\[0.8ex]
		R5:= PD^{-1}\psi_{*}\left([S^{[1, 2]}]\cap \ch(\mathcal{L}^{\vee})\cdot \sigma^*(\ch(\Ol_{\Delta}))\cdot \varphi^*(x) \cdot \rho^*(\alpha)\right), \\[0.8ex]
		R6:= PD^{-1}\psi_{*}\left([S^{[1, 2]}]\cap \varphi^*(x)\cdot \rho^*(\ch(2\Ol_S-\mathcal{T}_S)\cdot x)\right).
	\end{array}
\end{equation*}
Recall that we have taken $x=\q_1(1)|0\rangle$ and $\alpha=1$. We can now compute $c_2(S^{[2]})$ in terms of Nakajima operators.
\begin{prop} \label{prop c2(X) con operatori Nakajima}
	Let $S$ be a projective K3 surface. Then the second Chern class $c_2(S^{[2]}) \in H^{2,2}(S^{[2]}, \Z)$ of $S^{[2]}$ in terms of Nakajima operators is
	\begin{equation} \label{eq c_2(X)}
		c_2(S^{[2]})=27 \q_1(1)\q_1(x)|0\rangle +3 \displaystyle\sum_{i<j}\mu_{i,j}\q_1(\alpha_i)\q_1(\alpha_j)|0\rangle + \frac{3}{2} \displaystyle\sum_i \mu_{i,i}\q_1(\alpha_i)^2|0\rangle,
	\end{equation}
	where $1 \in H^0(S, \Z)$ is the unit and $x \in H^4(S, \Z)$ is the class of a point, $\{\alpha_1, \dots , \alpha_{22}\}$ is the basis of the lattice $H^2(S, \Z)$ used in Lemma $\ref{lem calcolo mu_ij}$ and the $\mu_{i,j}$'s are given in Table $\ref{tabella mu}$.
\end{prop}
Note that by Table $\ref{tabella mu}$ the integers $\mu_{i,i}$ are all even, so (\ref{eq c_2(X)}) is an element of $H^{2,2}(S^{[2]}, \Z)$.
\begin{proof}
	We make some computations on $L1, R1, \dots , R6$ introduced above.
	\begin{itemize}
		\item We have $L1=\ch(\mathcal{T}_2)\cdot \q_1(\alpha)(x)=\ch(\mathcal{T}_2)\cdot \q_1(1)\q_1(1)| 0\rangle$. By the definition of exponential Chern character and by $(\ref{formula 1sn Nakajima})$, which gives $\q_1(1)\q_1(1)|0\rangle =2\cdot 1_{S^{[2]}}$, we obtain
		\begin{equation*}
			L1=8\cdot 1_{S^{[2]}}-2c_2(S^{[2]})+\frac{1}{6}c_2(S^{[2]})^2-\frac{1}{3}c_4(S^{[2]}).
		\end{equation*}
		\item Since $S$ is a K3 surface we have $c_1(S)=0$ and $c_2(S)=24x$, where $x \in H^4(S, \Z)$ is the class of a point on $S$. Hence $\ch(\mathcal{T}_S)=2-24x$ and we obtain
		\begin{equation*}
			\begin{array}{lll}
				R1 & = & 2 \q_1(1)\q_1(1)|0\rangle - 24 \q_1(1)\q_1(x)| 0 \rangle \\[1ex]
				& = & 4\cdot 1_{S^{[2]}}-24\q_1(1)\q_1(x)|0\rangle,
			\end{array}
		\end{equation*}
		where the second equality comes from $(\ref{formula 1sn Nakajima})$.
		\item Let $d:=c_1(\mathcal{L})$. Then we have $\ch(\mathcal{L})=\sum_{\nu \ge 0} \frac{1}{\nu!}d^{\nu}$. Now, \cite[Lemma 3.9]{lehn1999chern} implies that the cycle $[S^{[1, 2]}] \cap d^{\nu}$ induces the operator $\q_1^{(\nu)}$, as observed in the proof of \cite[Lemma 4.2]{lehn1999chern}, hence we obtain
		\begin{equation*}
				R2 = \displaystyle\sum_{\nu \ge 0} \frac{1}{\nu!}\q_1^{(\nu)}(\alpha) \cdot x = \displaystyle\sum_{\nu \ge 0} \frac{1}{\nu!}\q_1^{(\nu)}(1)\q_1(1)|0 \rangle. 
		\end{equation*}
		We now compute $\q_1^{(\nu)}(1)\q_1(1)| 0 \rangle$ for every $\nu \ge 0$. If $\nu=0$, we have
		\begin{equation*}
			\q_1^{(0)}(1)\q_1(1)|0\rangle=\q_1(1)\q_1(1)|0\rangle=2\cdot 1_{S^{[2]}} \in H^0(S^{[2]}, \Z).
		\end{equation*} 
		If $\nu=1$, by Theorem $\ref{thm main Lehn}$ we have
		\begin{equation} \label{eq doppio tilda}
			\q_1^{\prime}(1)\q_1(1)|0\rangle = -\q_2(1)|0\rangle \in H^2(S^{[2]}, \Z).
		\end{equation} 
		If $\nu=2$, we have
		\begin{equation*}
			\begin{array}{lll}
				\q_1^{(2)}(1)\q_1(1)|0\rangle  & = & \left( \partial \q_1^{\prime}-\q_1^{\prime} \partial \right) (1)\q_1(1)|0 \rangle \\[1ex]
				& = & \partial \q_1^{\prime}(1)\q_1(1)|0\rangle -\q_1^{\prime}\partial(1)\q_1(1)|0\rangle.
			\end{array}
		\end{equation*}
		The boundary of $S$ is empty by Remark $\ref{rmk boundary empty}$, so $\q_1^{\prime}\partial(1)\q_1(1)|0 \rangle =0$. Moreover, using $(\ref{eq doppio tilda})$, we get $\partial\q_1^{\prime}(1)\q_1(1)|0\rangle =-\partial\q_2(1)|0\rangle$, and by Definition \ref{Defn boundary op} we obtain
		\begin{equation*}
			\q_1^{(2)}(1)\q_1(1)|0\rangle = \frac{1}{2}\q_2(1)|0\rangle \cdot \q_2(1)|0\rangle \in H^4(S^{[2]}, \Z).
		\end{equation*}
		Similarly for $\nu=3$ and $\nu=4$ we obtain the following:
		\begin{equation*}
			\begin{array}{lll}
				\q_1^{(3)}(1)\q_1(1)| 0 \rangle = -\frac{1}{2}\q_2(1)|0\rangle \cdot \frac{1}{2}\q_2(1)|0\rangle \cdot \q_2(1)|0\rangle \in H^6(S^{[2]}, \Z), \\[1ex]
				\q_1^{(4)}(1)\q_1(1)|0\rangle = \frac{1}{2}\q_2(1)|0\rangle \cdot \frac{1}{2}\q_2(1)|0\rangle \cdot \frac{1}{2} \q_2(1)|0\rangle \cdot \q_2(1)|0\rangle \in H^8(S^{[2]}, \Z).
			\end{array}
		\end{equation*}
		If $\nu \ge 5$, we obtain an element in $H^{2\nu}(S^{[2]}, \Z)=0$. We conclude that
		\begin{equation*}
			\begin{array}{lll}
				R2 & = & 2 \cdot 1_{S^{[2]}}-\q_2(1)|0\rangle +\frac{1}{2}\q_2(1)|0\rangle \cdot \frac{1}{2}\q_2(1)|0\rangle \\[1ex]
				& & - \frac{1}{3}\left( \frac{1}{2}\q_2(1)| 0 \rangle \cdot \frac{1}{2}\q_2(1)|0\rangle \cdot \frac{1}{2}\q_2(1)|0\rangle \right) \\[1ex]
				& & + \frac{1}{12}\left( \frac{1}{2}\q_2(1)|0\rangle \cdot \frac{1}{2}\q_2(1)|0\rangle \cdot \frac{1}{2}\q_2(1)|0\rangle \cdot \frac{1}{2}\q_2(1)|0\rangle \right).
			\end{array}
		\end{equation*}
		\item As before, we set $d:=c_1(\mathcal{L})$. By Lemma $\ref{lem duali formula EGL}$ we have $\ch(\Ol_{\Delta})=\ch(\Ol_{\Delta}^{\vee})$, and by Lemma $\ref{lem ch(O_Delta)}$ we have $\ch(\Ol_{\Delta})=[\Delta]-2y$, where $y \in H^8(S \times S, \Z)$ is the class of a point. Moreover, $[\Delta]=\sum_{i, j} \mu_{i,j}\alpha_i \otimes \alpha_j + 1 \otimes x + x \otimes 1$ by Lemma $\ref{lem calcolo mu_ij}$, where $\{\alpha_1, \dots , \alpha_{22}\}$ is the basis of $H^2(S, \Z)$ used in Lemma $\ref{lem calcolo mu_ij}$ and the $\mu_{i,j}$'s are given in Table $\ref{tabella mu}$. Recall the notation of diagram $(\ref{diagram EGL n=1})$. By the Künneth theorem we have $y=x \otimes x$. Then
		\begin{equation*}
			\begin{array}{lll}
				\sigma^*(\ch(\Ol_{\Delta})) & = & \sigma^*\left( \displaystyle\sum_{i, j} \mu_{i,j}\alpha_i \otimes \alpha_j + 1 \otimes x + x \otimes 1 - 2 (x \otimes x) \right) \\[1.5em]
				& = & \displaystyle\sum_{i,j}\mu_{i,j} \varphi^*(\alpha_i) \cdot \rho^*(\alpha_j) + \varphi^*(1)\cdot \rho^*(x)+\varphi^*(x)\cdot \rho^*(1)-2\left(\varphi^*(x)\cdot \rho^*(x)\right).
			\end{array}
		\end{equation*}
		Proceeding as for $R2$ we get
		\begin{equation} \label{eq R3}
			\begin{array}{lll}
				R3 & = & \displaystyle\sum_{i,j} \mu_{i,j} \displaystyle\sum_{\nu \ge 0} \frac{1}{\nu!}\q_1^{(\nu)}(\alpha_i)\q_1(\alpha_j)|0 \rangle + \displaystyle\sum_{\nu \ge 0} \frac{1}{\nu!}\q_1^{(\nu)}(x)\q_1(1)|0 \rangle\\[1.5em]
				& & +\displaystyle\sum_{\nu \ge 0} \frac{1}{\nu!}\q_1^{(\nu)}(1)\q_1(x)|0\rangle -2\displaystyle\sum_{\nu \ge 0} \frac{1}{\nu!}\q_1^{(\nu)}(x)\q_1(x)|0\rangle.
			\end{array}
		\end{equation}
		We call $R3_{\nu=i}$ the component of $R3$ obtained by putting $\nu=i$ in $(\ref{eq R3})$ for $i \ge 0$. Using the commutativity rule given by Theorem $\ref{thm Nakajima}$ we have
		\begin{equation*}
			R3_{\nu=0}=\displaystyle\sum_{i,j}\mu_{i,j}\q_1(\alpha_i)\q_1(\alpha_j)|0\rangle +2\q_1(1)\q_1(x)|0 \rangle -2\q_1(x)\q_1(x)|0 \rangle. 
		\end{equation*}
		Note that $\q_1(\alpha_i)\q_1(\alpha_j)|0\rangle \in H^4(S^{[2]}, \Z)$ and $\q_1(x)\q_1(x)|0\rangle \in H^8(S^{[2]}, \Z)$. If $\nu \ge 3$ we obtain elements in $H^i(S^{[2]}, \Q)$ with $i \ge 10$, so these are equal to zero.
		We do not compute explicitly $R3_{\nu=1}$: we will see that this is not necessary. If $\nu=2$, using Definition \ref{Defn boundary op} we obtain, after some computations,
		\begin{equation*}
				R3_{\nu=2} = \displaystyle\sum_{i,j}\mu_{i,j} \frac{1}{2}\left( \frac{1}{2}\q_2(1)|0\rangle \cdot \frac{1}{2}\q_2(1)|0\rangle \right) \cdot \q_1(\alpha_i)\q_1(\alpha_j)|0\rangle + \frac{1}{2}\q_2(1)|0\rangle \cdot \frac{1}{2}\q_2(1)|0\rangle \cdot \q_1(1)\q_1(x)|0\rangle,
		\end{equation*}
		which is an element of $H^8(S^{[2]}, \Z)$. We conclude that
		\begin{equation*}
			\begin{array}{lll}
				R3 & = & \displaystyle\sum_{i,j} \mu_{i,j} \q_1(\alpha_i)\q_1(\alpha_j)|0\rangle +2\q_1(1)\q_1(x)|0 \rangle -2\q_1(x)\q_1(x)|0\rangle + R3_{\nu=1}\\[1.3em]
				& & +\frac{1}{2}\displaystyle\sum_{i,j}\mu_{i,j} \frac{1}{2}\q_2(1)|0\rangle \cdot \frac{1}{2}\q_2(1)|0\rangle \cdot  \q_1(\alpha_i)\q_1(\alpha_j)|0\rangle \\[1.3em]
				& & +\frac{1}{2}\q_2(1)|0\rangle \cdot \frac{1}{2}\q_2(1)|0\rangle \cdot \q_1(1)\q_1(x)|0\rangle.
			\end{array}
		\end{equation*}
		\item By Lemma $\ref{lem duali formula EGL}$ we have $\mathcal{L}^{\vee} \cong \text{Hom}(\mathcal{L}, \Ol_{S^{[1, 2]}})$. Hence if $d:=c_1(\mathcal{L})$ we have $\ch(\mathcal{L}^{\vee})=\sum_{\nu \ge 0} \frac{(-1)^{\nu}}{\nu!}d^{\nu}$, so\,\,$R4$ is computed in the same way as $R2$, with a change of sign for the components obtained when $\nu=1$ and $\nu=3$. We obtain
		\begin{equation*}
			\begin{array}{lll}
				R4 & = & 2 \cdot 1_{S^{[2]}}+\q_2(1)|0\rangle +\frac{1}{2}\q_2(1)|0\rangle \cdot \frac{1}{2}\q_2(1)|0\rangle \\[1ex]
				& & + \frac{1}{3}\left( \frac{1}{2}\q_2(1)| 0 \rangle \cdot \frac{1}{2}\q_2(1)|0\rangle \cdot \frac{1}{2}\q_2(1)|0\rangle \right) \\[1ex]
				& & + \frac{1}{12}\left( \frac{1}{2}\q_2(1)|0\rangle \cdot \frac{1}{2}\q_2(1)|0\rangle \cdot \frac{1}{2}\q_2(1)|0\rangle \cdot \frac{1}{2}\q_2(1)|0\rangle \right).
			\end{array}
		\end{equation*}
		\item By Lemma \ref{lem duali formula EGL} we have $\mathcal{L}^{\vee}\cong\text{Hom}(\mathcal{L}, \Ol_{S^{[1, 2]}})$ and $\Ol_{\Delta}^{\vee}=\Ol_{\Delta}$, so $R5$ is computed in the same way as $R3$, with a change of sign for the component $R3_{\nu=1}$, so we obtain
		\begin{equation*}
			\begin{array}{lll}
				R5 & = & \displaystyle\sum_{i,j} \mu_{i,j} \q_1(\alpha_i)\q_1(\alpha_j)|0\rangle +2\q_1(1)\q_1(x)|0 \rangle -2\q_1(x)\q_1(x)|0\rangle - R3_{\nu=1} \\[1.3em]
				& & +\frac{1}{2}\displaystyle\sum_{i,j}\mu_{i,j} \frac{1}{2}\q_2(1)|0\rangle \cdot \frac{1}{2}\q_2(1)|0\rangle \cdot \q_1(\alpha_i)\q_1(\alpha_j)|0\rangle \\[1.3em]
				& & +\frac{1}{2}\q_2(1)|0\rangle \cdot \frac{1}{2}\q_2(1)|0\rangle \cdot \q_1(1)\q_1(x)|0\rangle.
			\end{array}
		\end{equation*}
		\item Since $S$ is a K3 surface, we have $\ch(2\Ol_S-\mathcal{T}_S)=24x$, where $x \in H^4(S, \Z)$ is the class of a point. Then
		\begin{equation*}
			R6=24\q_1(1)\q_1(x)| 0 \rangle ,
		\end{equation*}
		where we have used $\q_1(x)\q_1(1)|0\rangle=\q_1(1)\q_1(x)|0\rangle$ from Theorem $\ref{thm Nakajima}$.
	\end{itemize}
	Thus formula $(\ref{eq Chern character T2 Nakajima R1 R6})$ with $x=\q_1(1)|0\rangle$ and $\alpha=1$ gives
	\begin{equation} \label{eq L1=}
		\begin{array}{lll}
			L1 & = & 8 \cdot 1_{S^{[2]}} \\[1ex]
			& & -52\q_1(1)\q_1(x)|0\rangle +2\left( \frac{1}{2}\q_2(1)|0\rangle \cdot \frac{1}{2}\q_2(1)|0\rangle \right) -2\displaystyle\sum_{i, j}\mu_{i,j} \q_1(\alpha_i)\q_1(\alpha_j)|0\rangle \\[1ex]
			& & +\frac{1}{6}\left( \frac{1}{2}\q_2(1)|0\rangle \cdot \frac{1}{2}\q_2(1)|0\rangle \cdot \frac{1}{2}\q_2(1)|0\rangle \cdot \frac{1}{2}\q_2(1)|0\rangle \right) \\[1ex]
			& &+4\q_1(x)\q_1(x)|0\rangle \\[1ex]
			& & -\displaystyle\sum_{i,j}\mu_{i,j} \frac{1}{2}\q_2(1)|0\rangle \cdot \frac{1}{2}\q_2(1)|0\rangle \cdot \q_1(\alpha_i)\q_1(\alpha_j)|0\rangle \\[1.3em] 
			& & -2 \left( \frac{1}{2}\q_2(1)|0\rangle \cdot \frac{1}{2}\q_2(1)|0\rangle \cdot \q_1(1)\q_1(x)|0\rangle \right), \\
		\end{array}
	\end{equation} 
	where 
	\begin{equation} \label{eq L1}
		L1=8\cdot 1_S^{[2]}-2c_2(S^{[2]})+\frac{1}{6}c_2(S^{[2]})^2-\frac{1}{3}c_4(S^{[2]}).
	\end{equation}
	We now impose equalities between elements belonging to $H^4(S^{[2]}, \Z)$ in the right-hand side of $(\ref{eq L1=})$ and $(\ref{eq L1})$. We obtain
	\begin{equation*}
		c_2(S^{[2]})=26\q_1(1)\q_1(x)|0\rangle -\frac{1}{2}\q_2(1)|0\rangle \cdot \frac{1}{2}\q_2(1)|0\rangle + \displaystyle\sum_{i,j} \mu_{i,j} \q_1(\alpha_i)\q_1(\alpha_j)|0\rangle.
	\end{equation*}
	Using the commutativity rule of Theorem $\ref{thm Nakajima}$ and Lemma $\ref{lem cup products Lehn-Sorger 3}$, we get
	\begin{equation*}
		c_2(S^{[2]})=27\q_1(1)\q_1(x)|0\rangle +3\displaystyle\sum_{i<j}\mu_{i,j} \q_1(\alpha_i)\q_1(\alpha_j)|0\rangle + \frac{3}{2} \displaystyle\sum_i \mu_{i,i} \q_1(\alpha_i)^2 |0\rangle, 
	\end{equation*}
	and we are done.
\end{proof}
	\section{Integral Hodge classes of type $(2,2)$ on Hilbert squares of any K3 surface} \label{section integral Hodge}
	In this section we compute a basis for the lattice $H^{2,2}(S^{[2]}, \Z)$ of integral Hodge classes of type $(2, 2)$ on the Hilbert square of a general projective K3 surface $S$ whose Picard group is known. Proposition \ref{prop c2(X) con operatori Nakajima} implies the following.
\begin{thm} \label{thm 2/5q con op Nakajima}
	Let $S$ be a projective K3 surface and $X=S^{[2]}$ be its Hilbert square. Consider $q_X^{\vee} \in H^{2,2}(X, \Q)$, the dual of the BBF quadratic form. Then
	\begin{equation} \label{eq 2/5qvee con op Nakajima}
		\frac{2}{5}q_X^{\vee} = 9\q_1(1)\q_1(x)|0\rangle + \displaystyle\sum_{i<j}\mu_{i,j}\q_1(\alpha_i)\q_1(\alpha_j)|0\rangle+\frac{1}{2}\displaystyle\sum_i\mu_{i,i}\q_1(\alpha_i)^2|0\rangle \in H^{2,2}(X, \Z),
	\end{equation}
	where $1 \in H^0(S, \Z)$ is the unit, $x \in H^4(S, \Z)$ is the class of a point, $\{\alpha_1, \dots , \alpha_{22}\}$ is the basis of $H^2(S,\Z)$ used in Lemma $\ref{lem calcolo mu_ij}$ and the $\mu_{i,j}$'s are the integers given in Table $\ref{tabella mu}$. Moreover, $\frac{2}{5}q_X^{\vee}$ is indivisible in $H^{2,2}(X, \Z)$ and
	\begin{equation*} 
		\frac{1}{8}\left( \delta^2+\frac{2}{5}q_X^{\vee} \right) = \q_1(1)\q_1(x)|0\rangle \in H^{2,2}(X, \Z).
	\end{equation*}
\end{thm}
\begin{proof}
	By Proposition $\ref{prop O'Grady c_2(X)}$ we have $\frac{6}{5}q_X^{\vee}=c_2(X)$, so Proposition $\ref{prop c2(X) con operatori Nakajima}$ implies $(\ref{eq 2/5qvee con op Nakajima})$. Moreover, taking the basis of $H^4(X, \Z)$ given by Theorem $\ref{thm Qin--Wang}$, we see that $\frac{2}{5}q_X^{\vee}$ is indivisible in $H^4(X, \Z)$, i.e., there is no $\alpha \in H^4(X, \Z)$ such that $n\alpha=\frac{2}{5}q_X^{\vee}$ for some integer $n \in \Z_{>1}$: by Theorem \ref{thm Qin--Wang} it suffices to find some $\mu_{i,j}$ which are coprime with $9$, the coefficient of $\q_1(1)\q_1(x)|0\rangle$ in $(\ref{eq 2/5qvee con op Nakajima})$, in Table $\ref{tabella mu}$.
	This implies that $\frac{2}{5}q_X^{\vee}$ is indivisible also in $H^{2,2}(X, \Z)$. Recall that by Lemma $\ref{lem cup products Lehn-Sorger}$ we have
	\begin{equation} \label{eq delta 2 nella dim}
		\delta^2=-\displaystyle\sum_{i<j}\mu_{i,j}\q_1(\alpha_i)\q_1(\alpha_j)|0\rangle -\frac{1}{2}\displaystyle\sum_i \mu_{i,i}\q_1(\alpha_i)^2|0\rangle - \q_1(1)\q_1(x)|0\rangle,
	\end{equation}
	thus from $(\ref{eq 2/5qvee con op Nakajima})$ and $(\ref{eq delta 2 nella dim})$ we obtain $\delta^2+\frac{2}{5}q_X^{\vee}=8\q_1(1)\q_1(x)|0\rangle \in H^{2,2}(X, \Z)$, which implies
	\begin{equation*}
		\frac{1}{8}\left( \delta^2+\frac{2}{5}q_X^{\vee}\right) =\q_1(1)\q_1(x)|0\rangle \in H^{2,2}(X, \Z).
	\end{equation*}
\end{proof}
Note that the element $\frac{1}{8}\left( \delta^2+\frac{2}{5}q_X^{\vee}\right)$ appears also in \cite[Lemma 4.3]{shen2016hyperkahler}. It is possible to verify that Theorem\,\,$\ref{thm 2/5q con op Nakajima}$ is consistent with the computation of $\ch(S^{[2]})$ given by the \textsc{Maude} program in \cite[$\S 11$]{boissiere2007generating}, based on results obtained by Boissière in \cite{boissiere2005chern}, see \cite[$\S 3.4.3$]{novario2021ths} for details.
\begin{rmk} \label{rmk odd lattice}
	If $S$ is a projective K3 surface, the lattice $H^{2,2}(S^{[2]}, \Z)$ is always an odd lattice: this follows from the product $\langle \frac{1}{8}\left(\delta^2+\frac{2}{5}q_X^{\vee}\right), \frac{1}{8}\left(\delta^2+\frac{2}{5}q_X^{\vee}\right) \rangle = 1$.
\end{rmk}
We can now pass to the main result of this paper. Let $S$ be a general projective K3 surface. In Theorem \ref{thm dim H2,2(X, Q) per K3 qualsiasi} we have given a basis of the vector space $H^{2,2}(S^{[2]}, \Q)$ for a general projective K3 surface $S$. Then in Theorem \ref{thm H^22(X, Z) prima versione} we have described a basis of the lattice $H^{2,2}(S^{[2]}, \Z)$ when $S$ is generic, i.e., general with Picard group of rank $r=1$. We now present a basis of the lattice $H^{2,2}(S^{[2]}, \Z)$ for \emph{any general} projective K3 surface $S$ with Picard group of rank $r$, where $1 \le r \le 19$, since $h^{1,1}(S)=20$ for a K3 surface and, as remarked in Section \ref{intro}, a K3 surface of Picard rank $20$ is not general. We will give both a basis in terms of Nakajima operators and a basis which does not depend on Nakajima operators. In the particular case of general K3 surfaces, the next theorem will prove the conjecture in Section \ref{section second Chern}. We will use results obtained in Section \ref{section second Chern} and Theorem \ref{thm 2/5q con op Nakajima}. 
\begin{thm} \label{thm H^2,2(X, Z) generico}
	Let $S$ be a general projective K3 surface and let $\{b_1, \dots , b_r\}$ be a basis of $\Pic(S)$. Then:
	\begin{enumerate}[label=(\roman*)]
		\item $\rk(H^{2,2}(S^{[2]}, \Z))=\frac{(r+1)r}{2}+r+2$.
		\item The lattice $H^{2,2}(S^{[2]}, \Z)$ is odd and a basis is given by the following elements:
		\begin{itemize}
			\item $\q_2(b_i)|0\rangle$, for $i=1, \dots , r$,
			\item $\q_1(1)\q_1(x)|0\rangle$, where $1 \in H^0(S, \Z)$ is the unit and $x \in H^4(S, \Z)$ is the class of a point.
			\item $\frac{1}{2}\left( \q_1(b_i)^2-\q_2(b_i)\right)|0\rangle$, for $i=1, \dots, r$,
			\item $\q_1(b_i)\q_1(b_j)|0\rangle$, for $1 \le i < j \le r$,
			\item $-\displaystyle\sum_{i<j}\mu_{i,j}\q_1(\alpha_i)\q_1(\alpha_j)|0\rangle-\frac{1}{2}\displaystyle\sum_i\mu_{i,i}\q_1(\alpha_i)^2|0\rangle-\q_1(1)\q_1(x)|0\rangle$, where $\{\alpha_1, \dots , \alpha_{22}\}$ is the basis of $H^2(S, \Z)$ used in Lemma $\ref{lem calcolo mu_ij}$ and the $\mu_{i,j}$'s are given in Table $\ref{tabella mu}$.
		\end{itemize}
	Equivalently, the following is a basis of $H^{2,2}(S^{[2]}, \Z)$:
	\begin{equation} \label{eq basis H2,2}
		\left\{b_ib_j, \frac{b_i^2-b_i\delta}{2}, \frac{1}{8} \left( \delta^2 +\frac{2}{5}q_X^{\vee}\right), \delta^2\right\}_{1 \le i \le j \le r.}
	\end{equation}
	\end{enumerate}
In particular, if $S=S_{2t}$ is a generic K3 surface of degree $2t$, and $h \in \Pic(S^{[2]}_{2t})$ is the class induced by the ample generator of $\Pic(S_{2t})$, then
\begin{equation} \label{eq H^2,2 generic K3}
	H^{2,2}(S^{[2]}_{2t}, \Z)=\Z h^2 \oplus \Z \frac{h^2-h\delta}{2} \oplus \Z \frac{1}{8}\left( \delta^2+\frac{2}{5}q_X^{\vee}\right) \oplus \Z \delta^2.
\end{equation}
Moreover, $\textnormal{disc}(H^{2,2}(S^{[2]}_{2t}, \Z))=84t^3$ and the Gram matrix in the basis given above is the following:
\begin{equation*}
	\begin{pmatrix}
		12t^2 & 6t^2 & 2t & -4t \\
		6t^2 & t(3t-1) & t & -2t \\
		2t & t & 1 & -1 \\
		-4t & -2t & -1 & 12
	\end{pmatrix}.
\end{equation*}
\end{thm}
\begin{proof}
	By Theorem $\ref{thm dim H2,2(X, Q) per K3 qualsiasi}$ we have $\Dim(H^{2,2}(S^{[2]}, \Q))=\frac{(r+1)r}{2}+r+2$. Since by \cite[Theorem 1]{markman2007integral} the cohomology groups $H^i(S^{[2]}, \Z)$ are torsion free for $i \ge 0$, we obtain $\rk(H^{2,2}(S^{[2]}, \Z))=\frac{(r+1)r}{2}+r+2$. Remark \ref{rmk odd lattice} shows that $H^{2,2}(S^{[2]}, \Z)$ is an odd lattice. After a slight modification of the basis given in Theorem $\ref{thm dim H2,2(X, Q) per K3 qualsiasi}$, we have that the following is a basis of $H^{2,2}(S^{[2]}, \Q)$:
	\begin{equation} \label{base tesi}
		\begin{array}{l}
			\q_2(b_i)|0\rangle\,\,\,\text{for}\,\,i=1, \dots , r \\[1ex]
			\q_1(1)\q_1(x)|0\rangle \\[1ex]
			\frac{1}{2}\left( \q_1(b_i)^2-\q_2(b_i)\right) |0\rangle \,\,\,\text{for}\,\,i=1, \dots , r \\[1ex]
			\q_1(b_i)\q_1(b_j)|0\rangle\,\,\,\text{for}\,\,1\le i<j\le r \\[1ex]
			-\displaystyle\sum_{i<j}\mu_{i,j}\q_1(\alpha_i)\q_1(\alpha_j)|0\rangle - \frac{1}{2}\displaystyle\sum_i \mu_{i,i} \q_1(\alpha_i)^2|0\rangle-\q_1(1)\q_1(x)|0\rangle.
		\end{array}
	\end{equation}
	The strategy of the proof is the following: we look for a sublattice $L$ of $H^4(S^{[2]}, \Z)$ of maximal rank such that $L \cap H^{2,2}(S^{[2]}, \Q)=H^{2,2}(S^{[2]}, \Z)$ and such that a basis of $L$ contains the elements in $(\ref{base tesi})$. Since the Picard group $\Pic(S)$ of $S$ can be primitively embedded in $H^2(S, \Z)$, there exists a basis of $H^2(S, \Z)$ of the form $\{b_1, \dots, b_r, b_{r+1}, \dots , b_{22}\}$ for some $b_{r+1}, \dots , b_{22} \in H^2(S, \Z)$. By Theorem $\ref{thm Qin--Wang}$ the following is a basis of $H^4(S^{[2]}, \Z)$:
	\begin{equation} \label{base QW b_i}
			\left\{\q_1(1)\q_1(x)|0\rangle, \q_2(b_i)|0\rangle, \q_1(b_i)\q_1(b_j)|0\rangle, \frac{1}{2}\left( \q_1(b_i)^2-\q_2(b_i)\right) |0\rangle\right\},
	\end{equation}
	where $i, j \in \{1, \dots , 22\}$ and $i<j$. Recall that by Lemma \ref{lem cup products Lehn-Sorger} we have
	\begin{equation} \label{eq wertyuiop}
		\delta^2=-\displaystyle\sum_{i<j}\mu_{i,j}\q_1(\alpha_i)\q_1(\alpha_j)|0\rangle -\frac{1}{2}\displaystyle\sum_i \mu_{i,i}\q_1(\alpha_i)^2|0\rangle -\q_1(1)\q_1(x)|0\rangle,
	\end{equation}
	where $\{\alpha_1, \dots, \alpha_{22}\}$ is the basis of the lattice $H^2(S, \Z)$ used in Lemma \ref{lem calcolo mu_ij} and the $\mu_{i,j}$'s are the integers given in Table \ref{tabella mu}. Using the same procedure of Lemma $\ref{lem calcolo mu_ij}$ and Lemma $\ref{lem cup products Lehn-Sorger}$ with the basis $\{b_1, \dots, b_{22}\}$, we obtain
	\begin{equation} \label{eq delta^2 triangolo}
		\delta^2=-\displaystyle\sum_{i<j}\sigma_{i,j}\q_1(b_i)\q_1(b_j)|0\rangle -\frac{1}{2}\displaystyle\sum_i \sigma_{i,i} \q_1(b_i)^2|0\rangle -\q_1(1)\q_1(x)|0\rangle
	\end{equation}
	for some integers $\sigma_{i, j}$. Then the $\mu_{i,j}$'s are associated to the basis $\{\alpha_1, \dots , \alpha_{22}\}$ by the description of $\delta^2$ in (\ref{eq wertyuiop}), and the $\sigma_{i,j}$'s are associated to the basis $\{b_1, \dots , b_{22}\}$ by the description of $\delta^2$ in (\ref{eq delta^2 triangolo}). 
	
	We show that there exist positive integers $l$ and $k$ with $l<k$ and $k \ge r+1$ such that $\sigma_{l, k} \neq 0$. Suppose by contradiction that $\sigma_{i,j}=0$ for every $(i, j)$ such that $i<j$ and $j \ge r+1$. Then $(\ref{eq delta^2 triangolo})$ becomes
	\begin{equation} \label{delta^2 eq tilde b_i}
		\delta^2=-\displaystyle\sum_{1 \le i < j \le r}\sigma_{i,j} \q_1(b_i)\q_1(b_j)|0\rangle-\frac{1}{2}\displaystyle\sum_{i=1}^{22} \sigma_{i, i} \q_1(b_i)^2|0\rangle -\q_1(1)\q_1(x)|0\rangle.
	\end{equation}
	Consider the transcendental lattice $T(S) \cong \NS(S)^{\perp}$. Let $x \in T(S)$. Since $H^2(S, \Z)$ is a unimodular lattice, there exists $y \in H^2(S, \Z)$ such that $\int_S xy=1$. By Proposition $\ref{prop rmk 2.1 OG}$ we have 
	\begin{equation} \label{eq 1 proof H^22 Z}
		\langle \delta^2, xy \rangle =-2.
	\end{equation}
	By $(\ref{delta^2 eq tilde b_i})$, Lemma $\ref{lem cup products Lehn-Sorger}$ and Theorem $\ref{eq 2/5qvee con op Nakajima}$ we have
	\begin{equation} \label{hhjjkk}
		\begin{array}{lll}
			\langle \delta^2, xy \rangle & = & \langle -\displaystyle\sum_{1\le i < j \le r}\sigma_{i, j}\left[ b_ib_j-\left( \displaystyle\int_Sb_ib_j\right)\left(\frac{1}{8}\delta^2+\frac{1}{20}q_X^{\vee}\right)\right] \\[1.3em]
			& & -\frac{1}{2}\displaystyle\sum_{i=1}^{22}\sigma_{i,i}\left[ b_i^2-\left(\displaystyle\int_Sb_i^2\right)\left( \frac{1}{8}\delta^2+\frac{1}{20}q_X^{\vee}\right)\right] \\[1.2em]
			& & -\left(\frac{1}{8}\delta^2+\frac{1}{20}q_X^{\vee}\right), xy \rangle.
		\end{array}
	\end{equation}
	Since $x \in T(S)$ we have $\int_S b_ix=0$ for $i=1, \dots , r$, hence the right-hand side of (\ref{hhjjkk}) is equal to
	\begin{equation*}
		\begin{array}{l}
			-\displaystyle\sum_{1 \le i < j \le r} \sigma_{i, j} \left[\displaystyle\int_Sb_ib_j+\frac{1}{4}\displaystyle\int_Sb_ib_j-\frac{5}{4}\displaystyle\int_Sb_ib_j\right] \\[1.3em]
			-\frac{1}{2}\displaystyle\sum_{i=1}^r\sigma_{i, i}\left[\displaystyle\int_Sb_i^2+\frac{1}{4}\displaystyle\int_Sb_i^2-\frac{5}{4}\displaystyle\int_Sb_i^2\right] \\[1.3em]
			-\frac{1}{2}\displaystyle\sum_{i=r+1}^{22}\sigma_{i,i}\left[\displaystyle\int_Sb_i^2+2\displaystyle\int_Sb_ix\displaystyle\int_Sb_iy+\frac{1}{4}\displaystyle\int_Sb_i^2-\frac{5}{4}\displaystyle\int_Sb_i^2\right].
		\end{array}
	\end{equation*}
	Note that $\int_Sb_ib_j+\frac{1}{4}\int_Sb_ib_j-\frac{5}{4}\int_Sb_ib_j=0$ and $\int_Sb_i^2+\frac{1}{4}\int_Sb_i^2-\frac{5}{4}\int_Sb_i^2=0$,
%	\begin{equation*}
%		\displaystyle\int_Sb_ib_j+\frac{1}{4}\displaystyle\int_Sb_ib_j-\frac{5}{4}\displaystyle\int_Sb_ib_j=0, \qquad \displaystyle\int_Sb_i^2+\frac{1}{4}\displaystyle\int_Sb_i^2-\frac{5}{4}\displaystyle\int_Sb_i^2=0,
%	\end{equation*}
	hence we finally obtain
	\begin{equation} \label{eq 2 proof H^22 Z}
		\langle \delta^2, xy \rangle = -\displaystyle\sum_{i=r+1}^{22}\sigma_{i,i} \displaystyle\int_Sb_ix\displaystyle\int_Sb_iy-1.
	\end{equation}
	Thus $(\ref{eq 1 proof H^22 Z})$ and $(\ref{eq 2 proof H^22 Z})$ imply 
	\begin{equation} \label{eq 3 proof H^22 Z}
		\displaystyle\sum_{i=r+1}^{22}\sigma_{i,i}\displaystyle\int_Sb_ix\displaystyle\int_Sb_iy=1.
	\end{equation}
	The $\sigma_{i,i}$'s are all even, otherwise by Theorem $\ref{thm Qin--Wang}$ the element $\delta^2$ in $(\ref{eq delta^2 triangolo})$ would not be integral, since a basis of $H^4(S^{[2]}, \Z)$ is given by (\ref{base QW b_i}). Hence the left-hand side of $(\ref{eq 3 proof H^22 Z})$ is even, so it cannot be equal to $1$. We get a contradiction, so there exist $l, k$ positive integers with $l < k$ and $k \ge r+1$ such that $\sigma_{l, k} \neq 0$. Let now $L$ be the sublattice of $H^4(S^{[2]}, \Z)$ with the following basis:
	\begin{equation} \label{basis eq proof H^2,2 Z doppio triangolo}
		\begin{array}{ll}
			(\text{i}) & \q_1(1)\q_1(x)|0\rangle, \\[1ex]
			(\text{ii}) & \q_1(b_i)\q_1(b_j)|0\rangle \,\,\,\text{with}\,\,1 \le i < j \le 22 \,\,\,\text{and}\,\,(i, j) \neq (l, k), \\[1ex]
			(\text{iii}) & \delta^2=-\displaystyle\sum_{i<j}\sigma_{i,j}\q_1(b_i)\q_1(b_j)|0\rangle -\frac{1}{2}\displaystyle\sum_i \sigma_{i,i}\q_1(b_i)^2|0\rangle - \q_1(1)\q_1(x)|0\rangle, \\[1.3em]
			(\text{iv}) & \q_2(b_i)|0\rangle\,\,\, \text{for}\,\,i=1, \dots , 22, \\[1ex]
			(\text{v}) & \frac{1}{2}\left( \q_1(b_i)^2-\q_2(b_i)\right)|0\rangle \,\,\,\text{for}\,\,i=1, \dots , 22.
		\end{array}
	\end{equation}
	Recall that by (\ref{eq delta^2 triangolo}) the element in (iii) is also equal to
	\begin{equation*}
		-\displaystyle\sum_{i<j}\mu_{i,j}\q_1(\alpha_i)\q_1(\alpha_j)|0\rangle -\frac{1}{2}\displaystyle\sum_i \mu_{i,i}\q_1(\alpha_i)^2|0\rangle - \q_1(1)\q_1(x)|0\rangle.
	\end{equation*}
	Since $\sigma_{l, k} \neq 0$, the elements in $(\ref{basis eq proof H^2,2 Z doppio triangolo})$ give a basis for $H^4(S^{[2]}, \Q)$, thus $L$ is a sublattice of $H^4(S^{[2]}, \Z)$ of maximal rank. If $\sigma_{l,k}=\pm 1$, then $\q_1(b_l)\q_1(b_k)|0\rangle$ can be obtained as an integral linear combination of elements in $(\ref{basis eq proof H^2,2 Z doppio triangolo})$, hence every element in the basis $(\ref{base QW b_i})$ of $H^4(S^{[2]}, \Z)$ is in $L$, so $L=H^4(S^{[2]}, \Z)$ and (\ref{base tesi}) is a basis of $H^{2,2}(S^{[2]}, \Z)$: that is what happens in the proof of Theorem \ref{thm H^22(X, Z) prima versione}. If $\sigma_{l, k} \neq \pm 1$, then $L \neq H^4(S^{[2]}, \Z)$. More precisely, we have
	\begin{equation} \label{eq H4(X, Z)/L}
		\frac{H^4(S^{[2]}, \Z)}{L}\cong \frac{\Z}{|\sigma_{l, k}|\Z}
	\end{equation}
	generated by $\q_1(b_l)\q_1(b_k)|0\rangle$. We show that
	\begin{equation*}
		L \cap H^{2,2}(S^{[2]}, \Q)=H^{2,2}(S^{[2]}, \Z).
	\end{equation*}
	The inclusion $L \cap H^{2,2}(S^{[2]}, \Q)\subseteq H^{2,2}(S^{[2]}, \Z)$ is clear. We now prove the inclusion $L \cap H^{2,2}(S^{[2]}, \Q) \supseteq H^{2,2}(S^{[2]}, \Z)$ by showing that if $z \not \in L \cap H^{2,2}(S^{[2]}, \Q)$ then $z \not \in H^{2,2}(S^{[2]}, \Z)$. If $z \not \in H^{2,2}(S^{[2]}, \Q)$ we are done. Suppose now that $z \not \in L$. Clearly we have $H^4(S^{[2]}, \Z) \cap H^{2,2}(S^{[2]}, \Q) = H^{2,2}(S^{[2]}, \Z)$. Since the quotient in (\ref{eq H4(X, Z)/L}) is generated by $\q_1(b_l)\q_1(b_k)|0\rangle$, it suffices to show that $\q_1(b_l)\q_1(b_k)|0\rangle \not \in H^{2,2}(S^{[2]}, \Q)$ to get the inclusion. Suppose by contradiction that $\q_1(b_l)\q_1(b_k)|0\rangle \in H^{2,2}(S^{[2]}, \Q)$. Hence $q_1(b_l)\q_1(b_k)|0\rangle$ is a rational linear combination of elements in (\ref{base tesi}). Recall that the last element in (\ref{base tesi}) can be written as
	\begin{equation*}
		\delta^2=-\displaystyle\sum_{i<j}\sigma_{i,j} \q_1(b_i)\q_1(b_j)|0\rangle-\frac{1}{2}\displaystyle\sum_{i} \sigma_{i, i} \q_1(b_i)^2|0\rangle -\q_1(1)\q_1(x)|0\rangle
	\end{equation*}
	Since $\q_1(b_l)\q_1(b_k)|0\rangle$ appears only in $\delta^2$ among the elements in (\ref{base tesi}), we see that $\q_1(b_l)\q_1(b_k)|0\rangle \in H^{2,2}(S^{[2]}, \Q)$ only if
	\begin{equation} \label{eq cond sui sigma}
		\sigma_{i,j}=0\,\,\text{for}\,\,i\le j, j \ge r+1 \,\,\text{and}\,\,(i, j) \neq (l, k).
	\end{equation}
	Let again $x \in T(S)$ and $y \in H^2(S, \Z)$ such that $\int_S xy=1$. Similarly to $(\ref{eq 2 proof H^22 Z})$ we have
	\begin{equation} \label{eq cond sui sigma parte 2}
		\langle \delta^2, xy \rangle= -\displaystyle\sum\limits_{\substack{i<j \\ j \ge r+1}}\sigma_{i,j} \left( \displaystyle\int_S b_ix \displaystyle\int_S b_jy+\displaystyle\int_S b_jx\displaystyle\int_S b_iy \right)-\frac{1}{2}\displaystyle\sum_{i=r+1}^{22}\sigma_{i,i}\left( 2\displaystyle\int_S b_ix \displaystyle\int_S b_iy \right) -1,
	\end{equation}
	which implies by $(\ref{eq 1 proof H^22 Z})$ the following:
	\begin{equation*} 
		\displaystyle\sum\limits_{\substack{i<j \\ j \ge r+1}} \sigma_{i, j}\left(\displaystyle\int_S b_ix \displaystyle\int_S b_jy+\displaystyle\int_S b_jx\displaystyle\int_S b_iy \right)+\displaystyle\sum_{i=r+1}^{22}\sigma_{i,i}\left( \displaystyle\int_S b_ix \displaystyle\int_S b_iy \right)=1.
	\end{equation*}
	Thus we see that (\ref{eq cond sui sigma}) is not true, otherwise (\ref{eq cond sui sigma parte 2}) becomes $\sigma_{l, k}\left(\int_S b_lx\int_S b_ky + \int_S b_kx  \int_S b_l y \right) =1$, which is false since by assumption $\sigma_{l, k} \neq \pm 1$. Hence there exists $(\alpha, \beta) \neq (l, k)$ with $\alpha \le \beta$ and $\beta \ge r+1$ such that $\sigma_{\alpha, \beta} \neq 0$. As remarked above, this shows that $\q_1(b_l)\q_1(b_k)|0\rangle \not \in H^{2,2}(S^{[2]}, \Q)$. Then
	\begin{equation*}
		L \cap H^{2,2}(S^{[2]}, \Q)=H^{2,2}(S^{[2]}, \Z).
	\end{equation*}
	We conclude that (\ref{base tesi}) is a basis of $H^{2,2}(S^{[2]}, \Z)$. Using Remark \ref{rmk delta and other line bundles in terms of Nakajima op}, Lemma \ref{lem cup products Lehn-Sorger} and Theorem \ref{thm 2/5q con op Nakajima}, the elements in (\ref{base tesi}) can be written as:
	\begin{equation*}
		\begin{array}{l}
			b_i\delta \qquad \text{for}\,\,i=1, \dots , r, \\[1ex]
			\frac{1}{8}\left( \delta^2 +\frac{2}{5}q_X^{\vee} \right), \\[1ex]
			\frac{1}{2}\left( b_i^2 -\displaystyle\int_S b_i^2 \cdot \frac{1}{8}\left( \delta^2 +\frac{2}{5}q_X^{\vee} \right) -b_i\delta\right) \qquad \text{for}\,\,i=1, \dots , r, \\[1.2em]
			b_ib_j-\displaystyle\int_S b_ib_j \cdot \frac{1}{8}\left( \delta^2+\frac{2}{5}q_X^{\vee} \right) \qquad \text{for}\,\,1 \le i < j \le r, \\[1ex]
			\delta^2.
		\end{array}
	\end{equation*}
We write $b_i$ both for the element in $\Pic(S)$ and for the line bundle that this induces on $\Pic(S^{[2]})$. By Remark \ref{rmk delta and other line bundles in terms of Nakajima op} we have $b_i=\q_1(1)\q_1(b_i)|0\rangle$ and $\delta=\frac{1}{2}\q_2(1)|0\rangle$ in $H^4(S^{[2]}, \Z)$. Then, starting from the basis of $H^{2,2}(S^{[2]}, \Z)$ found before, using Lemma \ref{lem cup products Lehn-Sorger} one can show that (\ref{eq basis H2,2}) is a basis of $H^{2,2}(S^{[2]}, \Z)$. In particular when $S=S_{2t}$ is a generic K3 surface of degree $2t$ we have (\ref{eq H^2,2 generic K3}): in this case the discriminant of $H^{2,2}(S^{[2]}_{2t}, \Z)$ and the Gram matrix can be computed using Proposition\,\,\ref{prop O'Grady int q_Xvee}.
	\end{proof}

	\newcommand{\etalchar}[1]{$^{#1}$}
	
\end{document}